\documentclass{siamart251216}

\usepackage{lipsum}
\usepackage{amsfonts}
\usepackage{graphicx}
\usepackage{cleveref}
\usepackage{epstopdf}
\usepackage{hyperref}
\usepackage{float}
\usepackage{yhmath}
\usepackage{tikz}
\usepackage{caption}
\usepackage{subcaption}
\usepackage{blkarray}
\usepackage{multirow}
\usepackage{xcolor}
\usepackage{colortbl}
\usepackage{rotating}
\usepackage{algorithm}
\usepackage{algpseudocode}  
\usepackage{booktabs}
\usepackage{blkarray} 

\ifpdf
  \DeclareGraphicsExtensions{.eps,.pdf,.png,.jpg}
\else
  \DeclareGraphicsExtensions{.eps}
\fi

\newcommand*{\trans}{^\top}

\newcommand{\diag}{\mathrm{diag}}

\newcommand{\R}{\mathbb{R}}
\newcommand{\C}{\mathbb{C}}
\newcommand{\N}{\mathbb{N}}

\newcommand{\dif}{\mathrm{d}}

\newcommand{\argmin}{\operatorname{argmin}}
\newcommand{\rmF}{\mathrm{F}}
\newcommand{\bfb}{\boldsymbol{b}}
\newcommand{\bfd}{\boldsymbol{d}}
\newcommand{\bfe}{\boldsymbol{e}}

\newcommand{\bfl}{\boldsymbol{l}}
\newcommand{\bfu}{\boldsymbol{u}}
\newcommand{\bfx}{\boldsymbol{x}}

\newcommand{\bfz}{\boldsymbol{z}}

\newcommand{\bfL}{\boldsymbol{L}}
\newcommand{\bfI}{\boldsymbol{I}}
\newcommand{\bfX}{\boldsymbol{X}}

\newcommand{\bfZ}{\boldsymbol{Z}}
\newcommand{\bfH}{\boldsymbol{H}}
\newcommand{\bfA}{\boldsymbol{A}}
\newcommand{\bfC}{\boldsymbol{C}}

\newcommand{\bfW}{\boldsymbol{W}}
\newcommand{\bfU}{\boldsymbol{U}}
\newcommand{\bfV}{\boldsymbol{V}}
\newcommand{\bfD}{\boldsymbol{D}}
\newcommand{\bfzero}{\boldsymbol{0}}

\newcommand{\bfLda}{\boldsymbol{\Lambda}}
\newcommand{\bfPhi}{\boldsymbol{\Phi}}

\newcommand{\eye}{\boldsymbol{I}}

\newcommand{\svd}{\texttt{svd}}
\newcommand{\eig}{\texttt{eig}}
\newcommand{\Imag}{\mathrm{Imag}}

\newcommand{\Sch}{{Schrödinger}}
\newcommand{\err}{\mathrm{err}}

\newcommand{\ee}[1]{\ensuremath{\times 10^{#1}}}

\newcommand{\mv}{\Delta M}
\newcommand{\ev}{\Delta E}

\newsiamremark{remark}{Remark}
\newsiamremark{hypothesis}{Hypothesis}
\crefname{hypothesis}{Hypothesis}{Hypotheses}
\newsiamthm{claim}{Claim}
\newsiamremark{fact}{Fact}
\crefname{fact}{Fact}{Facts}

\headers{Structure-Preserving Dynamic Mode Decomposition}{Y. Feng, W. Gao, and J. Yin}

\title{Structure-Preserving Dynamic Mode Decomposition for Highly Oscillatory Dynamics of Semiclassical Schrödinger Equations\thanks{Submitted to the editors DATE
\funding{Y. Feng and J. Yin were partially supported by Shanghai Rising-Star Program under Grant No. 24QA2700600 and National Natural Science Foundation of China under Grant No. 12571422.}}}

\author{Yizhe Feng\thanks{School of Mathematical Sciences, Fudan University, Shanghai 200433, China
  (\email{yzfeng24@m.fudan.edu.cn}, \email{wggao@fudan.edu.cn}, \email{jiayin@fudan.edu.cn}).}
\and Weiguo Gao\footnotemark[2] \thanks{Shanghai Key Laboratory of Contemporary Applied Mathematics, Shanghai 200433, China.}
\and Jia Yin\footnotemark[2]
}

\ifpdf
\hypersetup{
  pdftitle={Structure-Preserving Dynamic Mode Decomposition for Highly Oscillatory Dynamics of Semiclassical Schrödinger Equations},
  pdfauthor={Y. Feng, W. Gao, and J. Yin},
}
\fi

\begin{document}

\maketitle

\begin{abstract}
We propose two novel data-driven dynamic mode decomposition (DMD)-type methods, the {\bfseries Crank--Nicolson DMD} and the {\bfseries semi-implicit DMD}, to predict the highly oscillatory dynamics of the semiclassical \Sch\ equations efficiently and accurately. Unlike many existing DMD-type methods which directly models the dynamics of the wave function, our approach is based on learning the \Sch\ operator while explicitly incorporating mass and energy conservation laws. This approach ensures physical fidelity and endows the resulting methods with built-in model order reduction capabilities, without the necessity for additional dimensionality-reduction preprocessing. An analysis of training and prediction errors are given for theoretical guarantees. Extensive numerical experiments demonstrate the noise robustness, computational efficiency, and transferability to other equations of the proposed methods.
\end{abstract}



\section{Introduction}
\label{sec:introduction}

The semiclassical \Sch\ equation with a small scaled Planck constant \(\varepsilon\) \((0 < \varepsilon \ll 1)\) has the form
\begin{equation}
\label{eq:schrodinger_equation_general}
\begin{gathered}
\imath \varepsilon \partial_t u^{\varepsilon}(\bfx, t) = - \frac{\varepsilon^2}{2} \Delta u^{\varepsilon}(\bfx, t) +  V(\bfx) u^{\varepsilon}(\bfx, t), \quad t \in \R,\,\bfx \in \R^d, \\
u^{\varepsilon}(\bfx, 0) = u_0(\bfx), \quad \bfx \in \R^d,
\end{gathered}
\end{equation}
where \(V(\bfx)\in\mathbb{R}\) is a prescribed electrostatic potential and \(u^{\varepsilon}(\bfx,t)\in\mathbb{C}\) denotes the wave function. Along the dynamics in the model~\eqref{eq:schrodinger_equation_general}, the total mass and energy of the system are conserved, i.e., for any $t \geq 0$,
\begin{equation}
\label{eq:mass_conservation}
    M(u^{\varepsilon}(\cdot, t)):=\int_{\R^d}|u^{\varepsilon}(\bfx, t)|^2 d \bfx = M(u^{\varepsilon}(\cdot, 0)), 
\end{equation}
and
\begin{equation}
\label{eq:energy_conservation}
    E(u^{\varepsilon}(\cdot, t)):=\int_{\R^d}\Bigl[\frac{\varepsilon^2}{2}|\nabla u^{\varepsilon}(\bfx, t)|^2+V(\bfx)|u^{\varepsilon}(\bfx, t)|^2\Bigr] d \bfx= E(u^{\varepsilon}(\cdot, 0)).
\end{equation}
This equation plays an important role in many problems of solid-state physics~\cite{kittelIntroductionSolidState2004,slaterElectronsPerturbedPeriodic1949}. 
Many numerical methods have been developed to solve this problem, including the finite difference methods~\cite{baoUniformErrorEstimates2012}, the exponential wave integrator method~\cite{baoUniformOptimalError2014}, and the time-splitting methods~\cite{baoTimeSplittingSpectralApproximations2002}.

In the semiclassical \Sch\ equation~\eqref{eq:schrodinger_equation_general}, the wave function \(u^{\varepsilon}\) propagates high-frequency oscillations with wavelength \(\mathcal{O}(\varepsilon)\) in both space and time. 
Therefore, many existing numerical methods require finer time steps $\tau$ and spatial grid sizes $h$ to maintain accuracy as $\varepsilon$ decreases~\cite{baoTimeSplittingSpectralApproximations2002, baoNumericalStudyTimesplitting2003}. This requirement significantly increases the dimensionality of the numerical discretization, leading to high computational and memory costs and the so-called curse of dimensionality. In contrast, the underlying dynamics may be low-dimensional~\cite{nathankutzAppliedKoopmanTheory2018}, which motivates the use of model order reduction techniques such as dynamic mode decomposition (DMD) to accelerate the numerical propagation.

DMD is a data-driven and model-free dimension-reduction technique proposed by Schmid~\cite{schmidDynamicModeDecomposition2010}. It constructs a low-dimensional linear dynamical model to approximate the evolution of observables within a high-dimensional (nonlinear) dynamical system. DMD has been extensively studied in fluid mechanics~\cite{budisicAppliedKoopmanism2012,colbrookMultiverseDynamicMode2023, mezicAnalysisFluidFlows2013}, and has also found applications in diverse fields such as video processing~\cite{grosekDynamicModeDecomposition2014}, multiresolution analysis~\cite{kutzMultiresolutionDynamicMode2016}, epidemiology~\cite{proctorDiscoveringDynamicPatterns2015}, and financial trading~\cite{huaUsingDynamicMode2016, mannDynamicModeDecomposition2016}. In recent years, 
many studies have applied DMD-type methods to predict dynamics of quantum systems governed by \Sch-type and related equations. Goldschmidt et al. proposed the bilinear DMD, which learns separate drift and control generators from data, thereby preserving the bilinear Hamiltonian structure of quantum control systems~\cite{goldschmidtBilinearDynamicMode2021}, while Kaneko et al. applied DMD to forecast long-time dynamics in strongly entangled quantum many-body systems~\cite{kanekoForecastingLongtimeDynamics2025}, and Yin et al. employed DMD and its high-order variant to analyze and predict non-equilibrium many-body dynamics described by the Kadanoff--Baym equations~\cite{yinAnalyzingPredictingNonequilibrium2023,YIN2022101843}. For systems with conservation laws, Baddoo et al. advanced the physics-informed DMD (piDMD) which incorporates the conservation laws to enhance robustness and predictive accuracy~\cite{baddooPhysicsinformedDynamicMode2023}. 

However, for the highly oscillatory dynamics of the semiclassical Schr\"{o}dinger equation, existing DMD methods may suffer from strong sensitivity to noise~\cite{dawsonCharacterizingCorrectingEffect2016, dukeErrorAnalysisDynamic2012}, limited effectiveness in model order reduction, high computational costs, and physical inconsistency. None of the existing methods is able to simultaneously address all of these challenges, resulting in inefficient or inaccurate performance.


To address the issue, we propose two novel DMD-type algorithms, termed as \textbf{Crank--Nicolson DMD (CN-DMD)} and \textbf{semi-implicit DMD (SI-DMD)}. In contrast to many existing DMD-type methods, which directly models the evolution of the wave function, our approach focuses on learning the underlying \Sch\ operator, with two schemes corresponding to different strategies for approximating the time derivative. This perspective enables the systematic incorporation of physical constraints. Moreover, it endows the resulting models with built-in model order reduction capabilities, without the need for additional dimensionality-reduction preprocessing as required in piDMD. The proposed CN-DMD and SI-DMD algorithms are closely related to the Crank--Nicolson finite difference (CNFD) method and the semi-implicit finite difference (SIFD) method, respectively. This close connection enhances the interpretability of the algorithms from a numerical-analysis perspective.

The rest of this paper is structured as follows. In Section~\ref{sec:preliminary}, we review the classical DMD and its variant, piDMD. Section~\ref{sec:algorithm_implementations} introduces the CN-DMD and SI-DMD algorithms and demonstrates their physical consistency and built-in model-order reduction capabilities. Error and computational complexity analyses are also provided. In Section~\ref{sec:numerical_examples}, we present extensive numerical experiments to evaluate the performance, robustness, and advantages of model order reduction of CN-DMD and SI-DMD compared to the classical DMD and piDMD. Finally, we conclude our work in Section~\ref{sec:conclusion}.

\section{Dynamic mode decomposition}
\label{sec:preliminary}
In this section, we briefly describe the ideas and the numerical implementations of the classical DMD and its variant, piDMD. Furthermore, we illustrate the differences between these two DMD algorithms through a toy model.

\subsection{Notations}
The following notations will be employed throughout this work. Vectors in \(\R^n\) and \(\C^n\) are denoted by bold lowercase letters, e.g., \(\bfx = [x_1, x_2, \dots, x_n]\trans\). Real and complex matrices in \(\R^{n \times m}\) and \(\C^{n \times m}\) are denoted by bold uppercase letters, e.g., \(\bfX\).  The identity matrix is always denoted as $\bfI$. The conjugate transpose of a matrix \(\bfX\) is denoted by \(\bfX^*\). The Euclidean (or $2$-) norm of an $n$-dimensional vector \(\bfx\) is defined as
\begin{equation*}
    \|\bfx\|_2 = \Bigl(\,\sum_{k=1}^n \left|x_k\right|^2\Bigr)^{1/2}.
\end{equation*}
The Frobenius norm of an $n$-by-$m$ matrix \(\bfX\) is defined as the square root of the sum of the absolute squares of its entries,
\begin{equation*}
    \|\bfX\|_{\rmF} = \Bigl(\,\sum_{i=1}^n \sum_{j=1}^m \left|\bfX_{i, j}\right|^2\,\Bigr)^{1/2}.
\end{equation*}
Infinite-dimensional linear operators are denoted by calligraphic letters, e.g., \(\mathcal{L}\) or \(\mathcal{A}\).

\subsection{The classical DMD}
\label{subsec:dmd}
Consider a continuous dynamical system described by an ordinary differential equation of the form
\begin{equation}
\label{eq:model_problem}
    \frac{\dif \bfx}{\dif t} = \mathbf{f}(\bfx(t)), \quad t \geq 0. 
\end{equation}
where \(\bfx(t) = [x_1(t), x_2(t), \ldots, x_n(t)]\trans \in \C^n\) is a time-dependent state variable, and \(\mathbf{f} \colon \C^n \rightarrow \C^n\) is a nonlinear function of \(\bfx\). We denote the observed data as \(\{\bfx_k\}_{k=0}^m\), where $m + 1$ denotes the number of snapshots and $\bfx_k:=\bfx(t_k)$ represents the state value vector at time $t_k = k \Delta t$ for $k=0, 1, \dots, m$. 

The goal of DMD is to construct a linear matrix $\bfL^\text{DMD}$ such that  
\begin{equation}
\label{eq:model_problem_discretization}
    \bfx_{k+1} \approx \bfL^{\text{DMD}} \bfx_k, \quad k \geq 0.
\end{equation}  
This strategy follows from the Koopman theory~\cite{budisicAppliedKoopmanism2012,colbrookMultiverseDynamicMode2023},which maps the nonlinear dynamical system into an infinite-dimensional linear dynamical system acting on observable functions. 
The approximation~\eqref{eq:model_problem_discretization} follows from the solution to the least-squares problem:  
\begin{equation}
\label{eq:DMD_lsq}
    \bfL^{\text{DMD}} \in \underset{\bfL \in \C^{n \times n}}{\argmin} \|\bfX_2 - \bfL\bfX_1\|_{\rmF},
\end{equation}  
where the data matrices are defined as 
\begin{equation}
\label{eq:DMD_data_X1_X2}
    \bfX_1 = \left[\bfx_0, \dots, \bfx_{m-1}\right], \quad
    \bfX_2 = \left[\bfx_1, \dots, \bfx_m\right].
\end{equation}
The solution to \eqref{eq:DMD_lsq} with the minimal Frobenius norm gives
\begin{equation}
\label{eq:DMD_matrix}
    \bfL^{\text{DMD}}=\bfX_2 \bfX_1^{\dagger}.
\end{equation}
The pseudoinverse \(\bfX_1^{\dagger}\) can be obtained from the singular value decomposition (SVD) of \(\bfX_1 = \bfU_{\bfX_1} \boldsymbol{\Sigma}_{\bfX_1} \bfV_{\bfX_1}^*\), where \(\bfU_{\bfX_1}\in \C^{n \times n}\), \(\boldsymbol{\Sigma}_{\bfX_1} \in \C^{n \times m}\) and \(\bfV_{\bfX_1}\in \C^{m \times m}\). Also, we denote \(\sigma_1 \geq \sigma_2 \geq \cdots \geq \sigma_n\) as the singular values of $\bfX_1$. For brevity, we omit the subscripts of $\bfU_{\bfX_1}$, $\boldsymbol{\Sigma}_{\bfX_1}$, $\bfV_{\bfX_1}$ and the superscript of $\bfL^{\text{DMD}}$ in the following discussion.

For many large-scale systems, the corresponding data matrix \(\bfX_1\) may have a low rank \(r \ll \min \{n, m\}\), or the singular values on the diagonal of \(\boldsymbol{\Sigma}\) decay rapidly. In these cases, we say that the underlying dynamics is low-dimensional and the dominant dynamic modes can be obtained by projecting \(\bfL\) onto the subspace spanned by the leading right singular vectors of \(\bfL\). We take\footnote{For simplicity, we adopt the same slicing notation as in MATLAB.}
\begin{equation*}
    \widetilde{\bfU} = \bfU(:, 1\!:\!r), \quad 
    \widetilde{\boldsymbol{\Sigma}} = \boldsymbol{\Sigma}(1\!:\!r, 1\!:\!r), \quad 
    \widetilde{\bfV} = \bfV(:, 1\!:\!r),
\end{equation*}
where the truncation rank \(r\) is defined as  
\begin{equation*}
    r = \max \left\{ i : \sigma_i > \texttt{tol} \cdot \sigma_1 \right\}
\end{equation*} 
with a prescribed tolerance \(\texttt{tol}\) (e.g., \(\texttt{tol} = 10^{-6}\)). 
Substituting \(\bfX_1 \approx \widetilde{\bfU} \widetilde{\boldsymbol{\Sigma}} \widetilde{\bfV}^*\) into~\eqref{eq:DMD_matrix} and projecting yields a rank-\(r\) approximation
\begin{equation*}
\widetilde{\bfL} =  \widetilde{\bfU}^* \bfL \,\widetilde{\bfU}
\approx
\widetilde{\bfU}^* \bigl(\bfX_2 \widetilde{\bfV} \widetilde{\boldsymbol{\Sigma}}^{-1} \widetilde{\bfU}^*  \bigr)
\widetilde{\bfU}=\widetilde{\bfU}^* \bfX_2 \widetilde{\bfV} \widetilde{\boldsymbol{\Sigma}}^{-1}.
\end{equation*}

To propagate the original system \eqref{eq:model_problem_discretization}, what remains to be solved is the eigenvalue problem \(\widetilde{\bfL} \bfW=\bfW \bfLda\) 
where \(\bfLda = \diag(\lambda_1,\,\dots,\,\lambda_r)\) is composed of the eigenvalues, and the columns of \(\bfW\) give the corresponding eigenvectors. To obtain spectral modes in the original state space of \(\C^n\), we perform the transformation
\(
    \bfPhi=\bfX_2 \widetilde{\bfV} \widetilde{\boldsymbol{\Sigma}}^{-1} \bfW .
\)
The columns of \(\bfPhi = [\phi_1,\dots,\phi_r]\) are called the DMD modes. Denote
\begin{equation*}
    \boldsymbol{\Omega}=
    \diag(\imath \,\omega_1^{\mathrm{DMD}},\,\dots,\, \imath\,  \omega_r^{\mathrm{DMD}}), \quad \omega_{\ell}^{\mathrm{DMD}}=-\imath \frac{\ln \lambda_{\ell}}{\Delta t}, \quad \ell=1, \ldots, r,
\end{equation*}
then the dynamics of \(\bfx\) can be expressed as
\begin{equation*}
    \bfx(t) \approx \sum_{\ell=1}^r \boldsymbol{\phi}_{\ell} \exp \left(i \omega_{\ell}^{\mathrm{DMD}} t\right) b_{\ell}=\bfPhi \exp (\Omega t) \bfb.
\end{equation*}
Specifically, the system dynamics at time \(t_k = k \Delta t\) for any nonnegative integer \(k\) can be approximated by  
\begin{equation*}
    \bfx(t_k) = \bfx_k \approx \bfPhi \bfLda^k \bfb.
\end{equation*}
In the expressions, the amplitude vector \(\bfb=\left[b_1, \ldots, b_r\right]\trans\) is taken as the projection of the initial value onto the DMD modes 
\(
    \bfb=\bfPhi^{\dagger} \bfx_0.
\)
This completes the procedures of DMD.

From the description of the DMD procedure, it is straightforward to observe that the major computational cost arises from the SVD, which requires \(\mathcal{O}\!\left(\min (m^2 n, m n^2)\right)\) operations. If we employ a truncated SVD with rank $r$ instead, the cost can be reduced to \(\mathcal{O}\!\left(nmr\right)\).

As mentioned in Section \ref{sec:introduction}, DMD is an equation-free, data-driven method. There is no need to know the exact form of the underlying dynamics function \(\mathbf{f}( \cdot )\) in \eqref{eq:model_problem}. By utilizing the data from first few time steps, this method enables the prediction of future system states. Moreover, focusing on the \(r\) dominant modes and frequencies rather than the full spectrum reduces the computational cost of the simulation. As a result, this method can be of great value for many complicated nonlinear and high-dimensional dynamical systems.

\subsection{The physics-informed DMD (piDMD)}
To introduce the methodology, we restrict our attention to the one-dimensional (\(d=1\)) semiclassical \Sch\ equation~\eqref{eq:schrodinger_equation_general} with periodic boundary conditions on the interval $(a,b)$, which reads
\begin{equation}
\label{eq:schrodinger_equation}
\begin{gathered}
\varepsilon u_t^{\varepsilon} - \imath \frac{\varepsilon^2}{2} u_{x x}^{\varepsilon} + \imath V(x) u^{\varepsilon}=0, \quad a<x<b, \quad t>0, \\
u^{\varepsilon}(x, t=0)=u_0^{\varepsilon}(x), \quad a \leq x \leq b, \\
u^{\varepsilon}(a, t)=u^{\varepsilon}(b, t), \quad u_x^{\varepsilon}(a, t)=u_x^{\varepsilon}(b, t), \quad t>0 .
\end{gathered}
\end{equation}
The wave function $u^{\varepsilon}$ is mass-preserving~\cite{baoTimeSplittingSpectralApproximations2002} and if we use time-splitting Fourier pseudospectral method to numerically solve the equation, the numerical solution preserves the discretized mass, that is, for any $t \geq 0$,  
\begin{equation*}
    \left\|\bfu^{\varepsilon}(t) \right\|_2 \approx \left\|\bfu^{\varepsilon}(0) \right\|_2,
\end{equation*}
where $\bfu^{\varepsilon}(t)$ denotes the numerical solution at equally spaced grid points within the wave propagation region at time $t$, and $\|\cdot\|_2$ denotes the Euclidean norm of a vector. 
Furthermore, the evolution operator mapping numerical results at any two adjacent time points can be modeled as a constant unitary matrix; that is, the discretized wave function $\bfu^{\varepsilon}(t_k)$ at time point $t_k = k\Delta t$ can be computed by the linear unitary evolution
\begin{equation*}
    \bfu^{\varepsilon}(t_{k+1}) = \bfL \, \bfu^{\varepsilon}(t_{k}), \quad \bfL \bfL^* = \bfL^* \bfL = \bfI.
\end{equation*}
The iteration formula of the time-splitting method for~\eqref{eq:schrodinger_equation}, and even for nonlinear \Sch\ equations, naturally possesses this structure~\cite{baoTimeSplittingSpectralApproximations2002, baoNumericalStudyTimesplitting2003}.

The physics-informed DMD (piDMD) framework is designed to learn such a unitary evolution matrix from data by solving a unitary constraint least square problem~\cite{baddooPhysicsinformedDynamicMode2023}. For consistency, we consider the general  model~\eqref{eq:model_problem} and data matrices $\bfX_1$, $\bfX_2$~\eqref{eq:DMD_data_X1_X2}. In piDMD, the linear operator \(\bfL \in \C^{n \times n}\) is obtained by solving the following constrained least-squares problem
\begin{equation}
\label{eq:piDMD_lsq}
    \bfL^{\text{piDMD}} \in \underset{\bfL \bfL^* = \bfI}{\argmin} \|\bfX_2 - \bfL\bfX_1\|_{\mathrm{F}}.
\end{equation}
It was derived by Schönemann~\cite{schonemannGeneralizedSolutionOrthogonal1966} that one of the solution is 
\begin{equation}
\label{eq:piDMD_matrix}
    \bfL^{\text{piDMD}} =\bfU_{\bfX_2 \bfX_1^*} \bfV_{\bfX_2 \bfX_1^*}^*,
\end{equation}
where \(\bfU_{\bfX_2 \bfX_1^*} \in \C^{n \times n}\) and \(\bfV_{\bfX_2 \bfX_1^*}^* \in \C^{n \times n}\) denotes the full left and right singular vector matrices of  \(\bfX_2 \bfX_1^* \in \C^{n \times n}\), respectively. The solution is unique if and only if $\bfX_2 \bfX_1^*$ has full rank. Once the evolution matrix $\bfL^{\text{piDMD}}$ is obtained, the system can be approximately propagated in the same manner as in classical DMD via
\begin{equation}
\label{eq:unitary_model_problem_discretization}
    \bfx_{k+1} \approx \bfL^{\text{piDMD}} \bfx_k, \quad k \geq 0.
\end{equation}

\begin{remark}
\label{rmk:toy_model}
To illustrate the relationship between the classical DMD and its variant piDMD, we analyze a toy model. We consider the data $\{\bfx_k\}_{k=0}^{m}$ generated by a unitary matrix
\begin{equation}
\label{eq:toy_model_data}
    \bfx_k = \sum_{i=1}^r b_i \boldsymbol{\phi}_i \lambda_i^k \in \C^n,
\end{equation}
where eigenvectors $\{\boldsymbol{\phi}_i\}_{i=1}^r$ is an othornormal set and eigenvalues $\{\lambda_i\}_{i=1}^r$ are mutually distinct and lie on the unit circle. We apply both the classical DMD and piDMD independently to this problem. We denote $\bfPhi_r = \left[\boldsymbol{\phi}_1, \dots, \boldsymbol{\phi}_r \right] \in \C^{n \times r}$, and $\bfPhi_r^{\bot} \in \C^{n \times (n-r)}$ as its orthogonal complement, the vector $\bfb_r = [b_1, \dots, b_r]\trans \in \C^r$, $\bfLda_r = \diag(\lambda_1, \dots, \lambda_r) \in \C^{r \times r}$, and the Vandermonde matrix 
\begin{equation*}
    \bfZ_r = \begin{bmatrix}
        1 & \lambda_1 & \dots & \lambda_1^{m-1} \\
        \vdots & \vdots & \ddots & \vdots \\
        1 & \lambda_r & \dots & \lambda_r^{m-1} \\ 
    \end{bmatrix} \in \C^{r \times m}.
\end{equation*}
Under this setting, data matrices $\bfX_1$ and $\bfX_2$~\eqref{eq:DMD_data_X1_X2} can be expressed as
\begin{equation}
\label{eq:toymodel_X1}
    \bfX_1 = [\bfx_0, \dots, \bfx_{m-1}] = \bfPhi_r \, \diag(\bfb_r) \, \bfZ_r,
\end{equation}
and
\begin{equation}
\label{eq:toymodel_X2}
    \bfX_2 = [\bfx_1, \dots, \bfx_m] = \bfPhi_r \, \bfLda_r \, \diag(\bfb_r) \, \bfZ_r.
\end{equation}

\paragraph{The classical DMD}
The evolution matrix in DMD is then obtained from \eqref{eq:DMD_matrix} as
\begin{equation}
\label{eq:toymodel_DMD_matrix}
    \bfL^{\rm DMD} = \bfPhi_r \, \bfLda_r \, \bfPhi_r^*.
\end{equation}

\paragraph{The piDMD}
Using decompositions \eqref{eq:toymodel_X1} and \eqref{eq:toymodel_X2}, we obtain the decomposition
\[\bfX_2 \bfX_1^* = (\bfPhi_r\, \bfLda_r) (\diag(\bfb_r) \, \bfZ_r)(\diag(\bfb_r) \, \bfZ_r)^* \bfPhi_r.\]
By taking the SVD of $\diag(\bfb_r) \, \bfZ_r$, we can derive the SVD of \(\bfX_2 \bfX_1^*\), and subsequently compute the evolution matrix in piDMD from \eqref{eq:piDMD_matrix} as
\begin{equation}
\label{eq:toymodel_piDMD_matrix}
    \bfL^{\rm piDMD} = \bfPhi_r \, \bfLda_r \, \bfPhi_r^* + \bfPhi_r^{\bot} \bigl(\bfPhi_r^{\bot}\bigr)^*.
\end{equation}
Comparing \eqref{eq:toymodel_DMD_matrix} and \eqref{eq:toymodel_piDMD_matrix}, we observe that in this toy model the additional term $\bfPhi_r^{\bot}(\bfPhi_r^{\bot})^*$ renders $\bfL^{\rm piDMD}$ unitary, while leaving the predicted solutions $\{\bfx_k\}_{k=0}^{m}$ unchanged.
\end{remark}

\begin{remark}
The orthogonal Procrustes problem \eqref{eq:piDMD_lsq} has the same solution as the total least squares problem when the solution is constrained to be orthogonal~\cite{arunUnitarilyConstrainedTotal1992}.
The total least-squares formulation is well known to exhibit enhanced robustness in the presence of noise. This equivalence therefore provides a theoretical explanation for the robustness of piDMD under noisy measurements. By contrast, the classical DMD lacks such structural constraints and is therefore more sensitive to noise \cite{dawsonCharacterizingCorrectingEffect2016,dukeErrorAnalysisDynamic2012}. Consequently, the classical DMD and its variant piDMD may perform significantly differently in certain extreme cases.
\end{remark}

\section{Algorithm design}
\label{sec:algorithm_implementations}
\subsection{Main idea}

As discussed in section~\ref{sec:preliminary}, piDMD requires computing the full SVD of an $n \times n$ matrix, the computational complexity of which scales as $\mathcal{O}(n^3)$. In addition, the prediction of future steps requires computing and storing the entire unitary evolution matrix, so the memory cost is at order $\mathcal{O}(n^2)$. Consequently, both the computational and storage requirements will become prohibitively expensive for high-dimensional systems. To reduce these costs, a common remedy is to project the data onto the leading proper orthogonal decomposition (POD) modes \cite{chatterjeeIntroductionProperOrthogonal2000} and construct the model in the reduced subspace \cite{baddooPhysicsinformedDynamicMode2023}. However, the resulting reduced operator is no longer unitary on the full space.

On the other hand, although the classical DMD exhibits strong performance in model-order reduction, it is well known to suffer from pronounced sensitivity to noise~\cite{ dawsonCharacterizingCorrectingEffect2016,dukeErrorAnalysisDynamic2012} and the absence of physical constraints such as mass conservation, which may result in structural instability and unphysical predictions. 


To enhance the robustness and enforce the physical constraints while maintaining the intrinsic reduced-order properties of classical DMD, we propose CN-DMD and SI-DMD. Unlike piDMD with POD preprocessing, these methods learn physically consistent approximations of the underlying \Sch\ operator by explicitly enforcing the Hermitian structure, which can implicitly induce a unitary evolution operator on the whole space. 



\subsection{The Crank--Nicolson DMD}
\subsubsection{The method}
For simplicity, we rewrite the first equation in~\eqref{eq:schrodinger_equation} as
\begin{equation}
\label{eq:Schordinger_equation_reduced}
    \imath \partial_t u = \mathcal{A} u,
\end{equation}
where \(\mathcal{A}\) denotes the \Sch\ operator. Our method uses trajectory data to derive a finite-dimensional approximation \(\bfA^{\text{CN}} \in \C^{n \times n}\) to this operator. 
Given state value data~\(\{\bfx_k\}_{k=0}^{m}\) (\(m \leq n\)), we construct the augmented data matrices  
\begin{equation}
\label{eq:CN_DMD_data_matrix}
    \bfX_1 = \bigl[\bfx_{\frac{1}{2}}, \cdots, \bfx_{m-1 + \frac{1}{2}}\bigr], \quad 
    \bfX_2 = \bigl[\imath \delta_t^+ \bfx_0, \cdots, \imath \delta_t^+ \bfx_{m-1}\bigr],
\end{equation}
where, for any $k \in \{0, 1, \dots, m-1\}$,
\begin{equation*}
    \bfx_{k + \frac{1}{2}}=\frac{\bfx_k + \bfx_{k+1}}{2}, 
    \quad
    \delta_t^+ \bfx_k = \frac{\bfx_{k+1} - \bfx_k}{\tau}.
\end{equation*}  
Here, $\delta_t^+$ denotes the forward finite difference approximation of the time derivative \(\partial_t\), and $\tau>0$ is the time step between two consecutive data points. To incorporate physical constraints, we construct the discretized \Sch\ operator $\bfA^{\text{CN}} \in \C^{n \times n}$ by solving the following this Hermitian Procrustes problem~\cite{highamSymmetricProcrustesProblem1988}:
\begin{equation}
\label{eq:Hermitian_Procrustes_problem}
    \bfA^{\text{CN}} \in \underset{\bfA = \bfA^*}{\operatorname{argmin}}\|\bfX_2 - \bfA \bfX_1\|_{\rmF},
\end{equation}
The minimal Frobenius norm solution is given by
\begin{equation}
\label{eq:CN_DMD_Hermite_matrix}
    \bfA^{\text{CN}} = \bfU_{\bfX_1} \bfH \bfU_{\bfX_1}^*,
\end{equation}
where the entries of \(\bfH\) are defined, for any \(1 \le i,j \le n\), by
\begin{equation}
\label{eq:CN_DMD_H}
\bfH_{i,j} = \overline{\bfH_{j,i}} =
\begin{cases}
\displaystyle\frac{\sigma_i \,\overline{\bfC_{ji}}+\sigma_j \,\bfC_{ij}}{\sigma_i^2+\sigma_j^2}, & \text{if } \sigma_i^2+\sigma_j^2 \neq 0, \\[0.6em]
0, & \text{otherwise},
\end{cases}
\end{equation}
with \(\bfC = \bigl[\bfU_{\bfX_1}^* \bfX_2 \bfV_{\bfX_1}, \bfzero\bigr] \in \C^{n \times n}\).
Here, \(\bfU_{\bfX_1},\, \bfV_{\bfX_1}\) and \(\{\sigma_i\}\) are defined as in Subsection~\ref{subsec:dmd}. As we will show below, the decomposition \eqref{eq:CN_DMD_Hermite_matrix} of $\bfA^{\text{CN}}$ offers a reduction in the model order. In addition to model order reduction and mass conservation, another advantage for imposing the Hermitian constraint is its reduced sensitivity to noise compared with the unconstrained model~\cite[Lemma B.1]{baddooPhysicsinformedDynamicMode2023}.


For brevity, we omit the superscript of $\bfA^{\text{CN}}$ onwards in this subsection. Furthermore, we observe that \eqref{eq:Hermitian_Procrustes_problem} yields the relation  
\begin{equation*}
    \imath\frac{\bfx_{k+1} - \bfx_k}{\tau} \approx 
    \bfA \frac{\bfx_{k+1} + \bfx_k}{2}.
\end{equation*}
By a simple transformation, this leads to a recurrence of the Crank--Nicolson type which serves as the main iterative formula of our Crank--Nicolson DMD (CN-DMD) algorithm
\begin{equation}
\label{eq:CN_DMD_iteration}
    \bfx_{k+1} = \Bigl(\eye + \imath \frac{\tau}{2} \bfA \Bigr)^{-1} \Bigl(\eye - \imath\frac{\tau}{2} \bfA \Bigr) \bfx_{k},
\end{equation}
where $\bfx_k \in \C^n$ denotes the approximated solution for the $k$th-step and $\eye \in \C^{n \times n}$ denotes the identity matrix. Remarkably, this method is unconditionally stable, due to its conservation of discretized mass and energy.

\begin{theorem}
The CN-DMD preserves both the discretized mass and the discretized energy. In particular,
\begin{equation}\label{eq:mass}
    \|\bfx_k\|_2 = \|\bfx_0\|_2, \quad 
    \text{for any } k \in \N.
\end{equation}
Moreover, the discretized quadratic energy functional
\begin{equation}
\label{eq:discretized_energy}
    E(\bfx) := \bfx^* \bfA \bfx,
\end{equation}
which is a finite-dimensional discretization associated with the Hermitian operator \(\bfA\) of the continuous energy functional
\(\langle u^{\varepsilon}, \mathcal{A} u^{\varepsilon} \rangle_{L^2}\) in \eqref{eq:energy_conservation}, is conserved, i.e.,
\[
    E(\bfx_k) = E(\bfx_0), \quad 
    \text{for any } k \in \N.
\]
\end{theorem}
\begin{proof}
    The mass conservation~\eqref{eq:mass} follows directly from the fact that the evolution matrix 
    \begin{equation}
    \label{eq:CN_DMD_evolution_matrix}
        \bfL = \Bigl(\eye + \imath \frac{\tau}{2} \bfA \Bigr)^{-1}\Bigl(\eye - \imath \frac{\tau}{2} \bfA \Bigr).
    \end{equation}
    is unitary.
    The energy conservation is established through the relation
    \begin{equation}
    \label{eq:CN_energy_preserving_iteration}
    \begin{aligned}
        E(\bfx_{k+1}) &= E(\bfL \bfx_{k}) 
        = \bfx_k^* \, \bfL^* \bfA \, \bfL \bfx_k 
        = \bfx_k^* \, \bfA \left(\bfL^* \bfL\right) \bfx_k \\
        &= \bfx_k^* \, \bfA \bfx_k = E(\bfx_k),
    \end{aligned}
    \end{equation}
    where the first equality follows from the iteration \eqref{eq:CN_DMD_iteration}, while the third and fourth equalities follow from \eqref{eq:CN_DMD_evolution_matrix} together with the Hermitian property of $\bfA$.
\end{proof}

\begin{remark}
We remark here that stronger constraints on the operator may also be integrated, but they often face implementation challenges. For instance, imposing a positive-definite condition on the operator precludes an analytical solution to the constrained least-squares problem. As a result, we have to use iterative optimization methods such as gradient descent, which increases the computational cost~\cite{gillisSemianalyticalApproachPositive2018}. Alternatively, while the tridiagonal structure aligns more closely with the second-order central finite difference method and~\eqref{eq:Hermitian_Procrustes_problem} still admits explicit solutions under this condition~\cite{baddooPhysicsinformedDynamicMode2023}, it is difficult to reconcile with model order reduction techniques. Thus, the computational benefits of reduced order models cannot be effectively realized under this constraint. Consequently, our approach strikes a balance between computational efficiency and physical fidelity.
\end{remark}

\subsubsection{Error analysis}
\label{subsec:error_analysis}
In this subsection, we present an error analysis of the CN-DMD method. To simplify the notation, we denote \( \bfu_k \in \C^n\) ($k\geq 0$) as the true wave function evaluated at equally spaced spatial grid points at time $t_k = k\tau$. We consider the finite-dimensional system
\begin{equation}
\label{eq:sch_eq_CNFD}
    \imath \frac{\bfu_{k+1} - \bfu_k}{\tau} = \widehat{\bfA}\, \frac{\bfu_{k+1} + \bfu_k}{2} + \bfb_k, 
    \quad 
    k \ge 0,
\end{equation}
obtained by applying the CNFD discretization with periodic boundary condition to the original semiclassical \Sch\ equation, where the Hermitian matrix $\widehat{\bfA} \in \C^{n \times n}$ serves as the corresponding discretized \Sch\ operator and $\bfb_k$ collects truncation errors, external noise, and other perturbations.

The CN-DMD scheme is constructed by solving the Hermitian-constrained least-squares problem
\begin{equation}
\label{eq:training_residuals}
    \underset{\bfA=\bfA^*}{\min} \sum_{k=0}^{m-1} \bigl\| \bfl_k(\bfA) \bigr\|_2^2,
    \quad \text{where } 
    \bfl_k(\bfA) := \imath \frac{\bfu_{k+1}-\bfu_k}{\tau} - \bfA \frac{\bfu_{k+1}+\bfu_k}{2}.
\end{equation}
Specifically, $\bfl_k(\widehat{\bfA})$ represents the truncation error of the CNFD discretization at time~$t_k$. Define the approximation error as \(\bfe_k = \bfu_k - \bfx_k\) for \(k \ge 0\). Based on the training residuals, we derive a \emph{training error} bound over the given time horizon.

\begin{theorem}
For any \(0 \le k \le m-1\), we have
\begin{equation}\label{eq:posteriori_estimate}
    \bigl\|\bfe_{k+1}\bigr\|_2 \le \tau \sum_{i=0}^k \bigl\|\bfl_i(\bfA^{\rm{CN}})\bigr\|_2,
\end{equation}
where $\bfl_i(\bfA^{\rm{CN}})$, for $i \in \{0, 1, \dots, k\}$, denotes the training residuals~\eqref{eq:training_residuals} and $\bfA^{\rm{CN}}$ denotes the discretized \Sch\ operator arising from CN-DMD. 
\end{theorem}
\begin{proof}
To simplify notations, we omit the superscript of $\bfA^{\rm{CN}}$. Subtracting
\begin{equation*}
    \imath \frac{\bfu_{j+1}-\bfu_j}{\tau} - \bfA \frac{\bfu_{j+1}+\bfu_j}{2} = \bfl_j(\bfA)
\end{equation*}
from
\begin{equation}\label{eq:CN_DMD_iteration_err_analysis}
    \imath \frac{\bfx_{j+1}-\bfx_j}{\tau} = \bfA\frac{\bfx_{j+1}+\bfx_j}{2},
\end{equation}
yields
\begin{equation*}
    \imath \frac{\bfe_{j+1}-\bfe_j}{\tau} - \bfA \frac{\bfe_{j+1}+\bfe_j}{2} = \bfl_j(\bfA) .
\end{equation*}
Equivalently,
\begin{equation*}
    \bfe_{j+1} = \left(\eye + \imath \frac{\tau}{2}\bfA\right)^{-1} \left(\eye - \imath \frac{\tau}{2}\bfA \right)\bfe_k
    - \imath \tau \left(\eye + \imath \frac{\tau}{2}\bfA\right)^{-1}\bfl_j(\bfA) .
\end{equation*}
Since \(\bfA=\bfA^*\), the Cayley transform \(
\bfU = \bigl(\eye + \imath \tfrac{\tau}{2}\bfA\bigr)^{-1} \bigl(\eye - \imath \tfrac{\tau}{2}\bfA\bigr)
\) is unitary and
\(\bigl\|(\eye+\imath \tfrac{\tau}{2}\bfA)^{-1}\bigr\|_2\le 1\). Hence
\[
\bigl\|\bfe_{j+1}\bigr\|_2 \le \bigl\|\bfe_j\bigr\|_2 + \tau \bigl\|\bfl_j(\bfA)\bigr\|_2,
\]
and summing from \(j=0\) to \(k\) with \(\bfx_0=\bfu_0\) (so \(\|\bfe_0\|_2=0\)) yields~\eqref{eq:posteriori_estimate}. 
\end{proof}

More generally, the error sequence \(\{\bfe_k\}\) satisfies the following estimation, which can be interpreted as an upper bound on \emph{prediction error}.

\begin{theorem}
For any \(k \ge 0\), we have
\begin{equation}
\label{eq:priori_inequality}
    \bigl\|\bfe_{k+1}\bigr\|_2 \le \bigl\|\bfe_0\bigr\|_2 + \tau \Bigl(\,\sum_{i=0}^k \bigl\|\bfb_i\bigr\|_2 + k \,\bigl\|\bfA^{\rm{CN}} - \widehat{\bfA} \bigr\|_2 \,\bigl\|\bfu_0\bigr\|_2 \Bigr),
\end{equation}
where $\bfA^{\rm{CN}} \in \C^{n \times n}$ and $\widehat{\bfA} \in \C^{n \times n}$ denote the discretized \Sch\ operators arising from CN-DMD and CNFD, respectively. The error terms $\bfb_i \in \C^n$ are defined in~\eqref{eq:sch_eq_CNFD}.
\end{theorem}
\begin{proof}
To simplify notations, we omit the superscript of $\bfA^{\rm{CN}}$. Subtracting \eqref{eq:sch_eq_CNFD} from \eqref{eq:CN_DMD_iteration_err_analysis} yields
\begin{equation*}
    \imath \frac{\bfe_{j+1} - \bfe_j}{\tau} = \widehat{\bfA} \frac{\bfe_{j+1} + \bfe_j}{2} - (\bfA - \widehat{\bfA}) \frac{\bfx_{j+1} + \bfx_j}{2} + \bfb_j.
\end{equation*}
Equivalently,
\begin{equation*}
\begin{aligned}
    \bfe_{j+1}
    &= \left(\eye + \imath \frac{\tau}{2} \widehat{\bfA} \right)^{-1} \left(\eye - \imath \frac{\tau}{2} \widehat{\bfA}\right)\bfe_j  \\
    &\quad +  \imath \tau \left(\eye + \imath \frac{\tau}{2} \widehat{\bfA} \right)^{-1} (\bfA - \widehat{\bfA}) \frac{\bfx_{j+1} + \bfx_j}{2}
    - \imath \tau \left(\eye + \imath \frac{\tau}{2} \widehat{\bfA} \right)^{-1} \bfb_j .
\end{aligned}
\end{equation*}
Since $\widehat{\bfA}=\widehat{\bfA}^*$ is Hermitian, the Cayley transform 
\(
\bfU = \bigl(\eye + \imath \tfrac{\tau}{2}\widehat{\bfA}\bigr)^{-1}\!\bigl(\eye - \imath \tfrac{\tau}{2}\widehat{\bfA}\bigr)
\)
is unitary and 
\(
\bigl\|\bigl(\eye + \imath \tfrac{\tau}{2}\widehat{\bfA}\bigr)^{-1}\bigr\|_2 \le 1
\).
Applying the triangle inequality, we obtain
\begin{equation*}
    \bigl\|\bfe_{j+1}\bigr\|_2 \le \bigl\|\bfe_j\bigr\|_2 + \frac{\tau}{2}\bigl\|\bfA-\widehat{\bfA}\bigr\|_2 \bigl(\bigl\|\bfx_{j+1}\bigr\|_2 + \bigl\|\bfx_j\bigr\|_2\bigr) + \tau \bigl\|\bfb_j\bigr\|_2 .
\end{equation*}
Because $\bfA=\bfA^*$, the CN update preserves the 2\,-norm, so \(\|\bfx_j\|_2=\|\bfx_0\|_2\). With the same initial data \(\bfx_0=\bfu_0\), this gives \(\|\bfx_j\|_2=\|\bfu_0\|_2\) and hence
\begin{equation*}
    \bigl\|\bfe_{j+1}\bigr\|_2 \le \bigl\|\bfe_j\bigr\|_2 + \tau \bigl\|\bfb_j\bigr\|_2 + \tau \bigl\|\bfA-\widehat{\bfA}\bigr\|_2\, \bigl\|\bfu_0\bigr\|_2 .
\end{equation*}
Summing for \(j=0\) to \(k-1\) yields \eqref{eq:priori_inequality}.
\end{proof}

\subsubsection{Model order reduction}
To further improve the efficiency of the iterative scheme \eqref{eq:CN_DMD_iteration}, we employ a reduced order model to construct a low-rank approximation 
\begin{equation}
\label{eq:CN_DMD_Hermite_matrix_reduced}
    \widetilde{\bfA} = \widetilde{\bfU}_{\bfX_1} \,\widetilde{\bfH}\, \widetilde{\bfU}_{\bfX_1}^*,
\end{equation}
to $\bfA$, where
\begin{equation}
\label{eq:CN_DMD_H_reduced}
    \widetilde{\bfU}_{\bfX_1} = \bfU_{\bfX_1}(:,\,1\!:\!r) \in \C^{n \times r},
    \qquad
    \widetilde{\bfH} = \bfH(1\!:\!r,\,1\!:\!r) \in \C^{r \times r}.
\end{equation}
The numerical rank $r$ of \(\bfX_1\) is defined as  
\begin{equation*}
    r=\max \left\{i: \sigma_i> \texttt{tol} \cdot \sigma_1\right\},
\end{equation*}
with a user-supplied tolerance \(\texttt{tol}\) (e.g., \(\texttt{tol} = 10^{-6}\)). This truncation is discussed in more detail in~\cite{drmacHermitianDynamicMode2024}. A simple justification for this truncation is provided in Lemma~\ref{lemma:truncation_reason}, which shows that such a truncation is reasonable in our setting.

\begin{lemma}
\label{lemma:truncation_reason}
We consider the data $\{\bfx_k\}_{k=0}^{m}$ generated by a unitary matrix
\begin{equation}
\label{eq:toy_model_data_in_lemma}
    \bfx_k = \sum_{i=1}^r b_i \boldsymbol{\phi}_i \lambda_i^k \in \C^n,
\end{equation}
where eigenvectors $\{\boldsymbol{\phi}_i\}_{i=1}^r$ is an othornormal set and eigenvalues $\{\lambda_i\}_{i=1}^r$ are mutually distinct and lie on the unit circle. Then the Hermitian matrix \(\bfH\)~\eqref{eq:CN_DMD_H} can be partitioned into the block form:
\begin{equation}
\label{eq:CN_DMD_H_block_form}
\bfH =
\begin{blockarray}{c@{\hskip 15pt}@{\hskip 6pt}cc}
    & r & n-r \\
    \begin{block}{r[@{\hskip 6pt}cc]}
        r   & \bfH(1\!:\!r, 1\!:\!r) & \bfzero \\
        n-r & \bfzero     & \bfzero \\
    \end{block}
\end{blockarray}.
\end{equation}
\end{lemma}

\begin{proof}
 With the expression \eqref{eq:toy_model_data_in_lemma}, the data matrix $\bfX_1$ and $\bfX_2$ \eqref{eq:CN_DMD_data_matrix} can be reformulated as 
\begin{equation*}
    \bfX_1 = \frac{1}{2}\bigl(\bfPhi_r (\bfLda_r + \eye)\, \diag(\bfb_r) \bfZ_r \bigr) , \quad
    \bfX_2 = \frac{\imath}{2} \bigl(\bfPhi_r (\bfLda_r - \eye)\, \diag(\bfb_r) \bfZ_r \bigr),
\end{equation*}
where variables $\bfPhi_r$, $\bfLda_r$, $\bfb_r$ and $\bfZ_r$ are defined in Remark~\ref{rmk:toy_model}. Without loss of generality, we assume that eigenvalues $\lambda_1 \sim \lambda_r$ are all different from one. Then we have
\begin{equation*}
    \Imag(\bfX_1) = \Imag(\bfX_2) = \Imag(\bfPhi_r) = \Imag(\widetilde{\bfU}_{\bfX_1})
\end{equation*}
where $\widetilde{\bfU}_{\bfX_1} \in \C^{n \times r}$ \eqref{eq:CN_DMD_H_reduced} is the truncated left singular vector matrix of $\bfX_1$. Thus, 
\begin{equation*}
    \bfC = \begin{bmatrix}
        \widetilde{\bfU}_{\bfX_1}^* \\
        \widetilde{\bfU}_{\bfX_1}^{\bot *}
    \end{bmatrix} \bfX_2 \bfV_{\bfX_1}
         = \begin{bmatrix}
        \widetilde{\bfU}_{\bfX_1}^* \bfX_2 \bfV_{\bfX_1} \\
        \bfzero
    \end{bmatrix}.
\end{equation*}
Substitute into the exact expression \eqref{eq:CN_DMD_H} and combine with the Hermitian property of \(\bfH\), it is noticed that $\bfH$ can be partitioned into the block form~\eqref{eq:CN_DMD_H_block_form}.
\end{proof}

Furthermore, we consider the eigendecomposition of the truncated Hermitian matrix 
\begin{equation*}
    \widetilde{\bfH} = \bfW_{\widetilde{\bfH}} \boldsymbol{\Lambda}_{\widetilde{\bfH}} \bfW_{\widetilde{\bfH}}^*,
    \quad 
    \bigl(\bfW_{\widetilde{\bfH}}, \, \boldsymbol{\Lambda}_{\widetilde{\bfH}} \in 
    \C^{r \times r}\bigr).
\end{equation*}
Substituting this expression into~\eqref{eq:CN_DMD_Hermite_matrix_reduced} yields the reduced eigendecomposition
\begin{equation*}
    \widetilde{\bfA} = \widetilde{\bfU}_{\bfX_1} \bigl( \bfW_{\widetilde{\bfH}} \boldsymbol{\Lambda}_{\widetilde{\bfH}} \bfW_{\widetilde{\bfH}}^* \bigr) \widetilde{\bfU}_{\bfX_1}^*
    = \bigl(\widetilde{\bfU}_{\bfX_1} \bfW_{\widetilde{\bfH}}\big) \boldsymbol{\Lambda}_{\widetilde{\bfH}} \bigl(\widetilde{\bfU}_{\bfX_1} \bfW_{\widetilde{\bfH}}\bigr)^*.
\end{equation*}
We define \(\bfU_{\widetilde{\bfA}} = \widetilde{\bfU}_{\bfX_1}\bfW_{\widetilde{\bfH}} \in \C^{n \times r}\) and further denote
\(
\bfW_{\widetilde{\bfA}} = \bigl[\bfU_{\widetilde{\bfA}},\,\bfU_{\widetilde{\bfA}}^{\perp}\bigr] \in \C^{n \times n}
\)
as the full eigenvector matrix of \(\widetilde{\bfA}\), where the columns of \(\bfU_{\widetilde{\bfA}}^{\perp} \in \C^{n \times (n-r)}\) form an orthonormal basis for the orthogonal complement of \(\bfU_{\widetilde{\bfA}}\). In particular, it satisfies
\begin{equation}
\label{eq:complement_U_A_identity}
    \bfU_{\widetilde{\bfA}} \bfU_{\widetilde{\bfA}}^* + 
    \bfU_{\widetilde{\bfA}}^{\perp} \bfU_{\widetilde{\bfA}}^{\perp *} = \eye.
\end{equation}
Consequently, the full eigendecomposition of \(\widetilde{\bfA}\) is given by
\begin{equation}
\label{eq:CN_DMD_Hermite_matrix_reduced_eig}
    \widetilde{\bfA} = \bfW_{\widetilde{\bfA}} \boldsymbol{\Lambda}_{\widetilde{\bfA}}
    \bfW_{\widetilde{\bfA}}^*,
    \quad 
    \boldsymbol{\Lambda}_{\widetilde{\bfA}} = \begin{bmatrix}
        \boldsymbol{\Lambda}_{\widetilde{\bfH}} &           \\
                                            & \bfzero
    \end{bmatrix}.
\end{equation}
By approximating $\bfA$ by $\widetilde{\bfA}$ in the original recurrence~\eqref{eq:CN_DMD_iteration}, the iteration can be rewritten, for any nonnegative integer $k$, as
\begin{equation*}
    \begin{aligned}
    \bfx_k &= \bfW_{\widetilde{\bfA}} \bigl(\eye + \imath \frac{\tau}{2} \boldsymbol{\Lambda}_{\widetilde{\bfA}} \bigr)^{-k}
    \bigl(\eye - \imath\frac{\tau}{2} \boldsymbol{\Lambda}_{\widetilde{\bfA}}\bigr)^k \bfW_{\widetilde{\bfA}}^* \bfx_0, \\
    &= \bfW_{\widetilde{\bfA}} \, \bfD_{\widetilde{\bfA}}^k \, \bfW_{\widetilde{\bfA}}^* \, \bfx_0,
    \end{aligned}
\end{equation*}  
where the diagonal matrix \(\bfD_{\widetilde{\bfA}}\) is given by 
\begin{align}
    & \bfD_{\widetilde{\bfA}} = \bigl(\eye + \imath \frac{\tau}{2} \boldsymbol{\Lambda}_{\widetilde{\bfA}} \bigr)^{-1} \bigl(\eye - \imath \frac{\tau}{2} \boldsymbol{\Lambda}_{\widetilde{\bfA}} \bigr) 
    = \begin{bmatrix}
    \bfD_{\widetilde{\bfH}} & \\
    & \eye 
    \end{bmatrix}, \\
\label{eq:CN_DMD_matrix_eigmat_reduced}
    & \qquad \bfD_{\widetilde{\bfH}} = \bigl(\eye + \imath \frac{\tau}{2} \boldsymbol{\Lambda}_{\widetilde{\bfH}} \bigr)^{-1} \bigl(\eye - \imath \frac{\tau}{2} \boldsymbol{\Lambda}_{\widetilde{\bfH}} \bigr).
\end{align}
Consequently, the iteration~\eqref{eq:CN_DMD_iteration} can be written in a form only dependent on $\boldsymbol{\Lambda}_{\widetilde{\bfH}}$ and $\bfU_{\widetilde{\bfA}}$, for any nonnegative integer $k$, as
\begin{equation*}
    \begin{aligned}
        \bfx_k &= \begin{bmatrix}
            \bfU_{\widetilde{\bfA}} & \bfU_{\widetilde{\bfA}}^{\bot}  
        \end{bmatrix}
        \begin{bmatrix}
            \bfD_{\widetilde{\bfH}}^k & \\
                  & \eye
        \end{bmatrix}
        \begin{bmatrix}
            \bfU_{\widetilde{\bfA}}^* \\ \bfU_{\widetilde{\bfA}}^{\bot *}
        \end{bmatrix} \bfx_0 \\
        &= \bfU_{\widetilde{\bfA}}   \,
           \bfD_{\widetilde{\bfH}}^k \,
           \bfU_{\widetilde{\bfA}}^* \, \bfx_0 + \bfU_{\widetilde{\bfA}}^{\bot} 
        \bfU_{\widetilde{\bfA}}^{\bot *} \bfx_0 \\
        &= \bfU_{\widetilde{\bfA}}   \,
           \bfD_{\widetilde{\bfH}}^k \,
           \bfU_{\widetilde{\bfA}}^* \,\bfx_0 + \left(\eye -\bfU_{\widetilde{\bfA}} \bfU_{\widetilde{\bfA}}^*\right) \bfx_0,
    \end{aligned}
\end{equation*}
where the identity \eqref{eq:complement_U_A_identity} is used in the last equality. It means that we can accomplish the prediction with computing the reduced eigendecomposition of \(\widetilde{\bfA}\), which is beneficial for reducing computational and storage costs. The following Algorithm \ref{algo:CN_DMD} describes the overall procedure of CN-DMD for the single time prediction.

\begin{algorithm}[htbp]
\caption{The Crank--Nicolson DMD (single time prediction)}
\label{algo:CN_DMD}
\begin{algorithmic}[1]
    \State {\bf Input:} Initial state $\bfx_0$, data matrices $\bfX_1,\,\bfX_2 \in \C^{n \times m}$ from \eqref{eq:CN_DMD_data_matrix}, tolerance $\texttt{tol}$ for truncated SVD, length $N$ of target horizon (with time step $\tau$).
    \State $[\bfU_{\bfX_1}, \boldsymbol{\Sigma}_{\bfX_1}, \bfV_{\bfX_1}] \gets \svd(\bfX_1, \texttt{tol})$  \hfill \% Truncated SVD with tolerance $\texttt{tol}$.
    \State Construct the Hermitian matrix $\widetilde{\bfH} \in \C^{r \times r}$ as in \eqref{eq:CN_DMD_H} and \eqref{eq:CN_DMD_H_reduced}.
    \State $[\bfW_{\widetilde{\bfH}},\, \boldsymbol{\Lambda}_{\widetilde{\bfH}}] \gets \eig(\widetilde{\bfH})$ \hfill \% Eigendecomposition of $\widetilde{\bfH}$.
    \State $\bfU_{\widetilde{\bfA}} \gets \widetilde{\bfU}_{\bfX_1} \bfW_{\widetilde{\bfH}}$  \hfill \% Lift reduced eigenvectors to the full space.
    \State We denote $\boldsymbol{\Lambda}_{\widetilde{\bfH}} = \diag(\lambda_1, \dots, \lambda_r)$, then 
    \begin{equation*}
        \bfD_{\widetilde{\bfH}} \gets \diag\biggl(\dfrac{1-\imath \tfrac{\tau}{2}\lambda_j}{\,1+\imath \tfrac{\tau}{2}\lambda_j}\biggr)_{j=1}^r
    \end{equation*}
    \hfill \% Crank--Nicolson spectral factors.
    \State Compute the predicted state $\bfx_N$ at time $N \tau$:
    \begin{equation*}
       \bfx_N = \bfU_{\widetilde{\bfA}} \, \bfD_{\widetilde{\bfH}}^N \, \bfU_{\widetilde{\bfA}}^* \, \bfx_0 + (\eye -\bfU_{\widetilde{\bfA}} \bfU_{\widetilde{\bfA}}^*) \bfx_0.     
    \end{equation*}
    \State {\bf Output:} Predicted state $\bfx_N$ at time $N \tau$.
\end{algorithmic}
\end{algorithm}

\subsection{The Semi-implicit DMD}
\subsubsection{The method}
The semi-implicit DMD (SI-DMD) differs from CN-DMD mainly in the choice of the finite difference scheme used to approximate the time derivative. The subsequent derivation of the iteration scheme and the principle of model order reduction follow in an  analogous manner.

We construct the augmented data matrices
\begin{equation}
    \label{eq:SI_DMD_data_matrix}
    \bfX_1 = \bigl[\hat{\bfx}_{1}, \cdots, \hat{\bfx}_{m}\bigr], \quad 
    \bfX_2 = \bigl[\imath \delta_t \bfx_1, \cdots, \imath \delta_t \bfx_m\bigr],
\end{equation}  
where 
\begin{equation*}
    \hat{\bfx}_k=\frac{\bfx_{k+1} + \bfx_{k-1}}{2}, \quad
    \delta_t \bfx_k = \frac{\bfx_{k+1} - \bfx_{k-1}}{2\tau}, \quad k=1,\cdots, m.
\end{equation*} 
Here $\delta_t$ denotes the second-order central finite difference approximation of the temporal derivative $\partial_t$. The following operations are similar to CN-DMD. We construct the approximated \Sch\ operator \(\bfA^{\text{SI}}\) by solving the Hermitian Procrustes problem
\begin{equation}
\label{eq:Hermitian_Procrustes_problem_SI}
    \bfA^{\text{SI}} \in \underset{\bfA = \bfA^*}{\operatorname{argmin}}\|\bfX_2 - \bfA \bfX_1\|_{\rmF},
\end{equation}
The minimum Frobenius norm solution is the same as~\eqref{eq:CN_DMD_H} in the form. We omit the superscript of $\bfA^{\text{SI}}$ for brevity. Further, \eqref{eq:SI_DMD_data_matrix} and \eqref{eq:Hermitian_Procrustes_problem_SI} yield the relation: for $1 \leq k \leq m-1$,
\begin{equation*}
    \imath\frac{\bfx_{k+1} - \bfx_{k-1}}{2\tau} \approx 
    \bfA \frac{\bfx_{k+1} + \bfx_{k-1}}{2}. 
\end{equation*}
By a simple transformation, this leads to a recurrence of the semi-implicit type,  
\begin{equation}
\label{eq:SI_DMD_recurrence}
    \bfx_{k+1} = \left(\eye + \imath \tau \bfA \right)^{-1}(\eye - \imath \tau \bfA) \bfx_{k-1}.
\end{equation}
This formula is served as the main iterative formula of our SI-DMD algorithm. It is unconditionally stable due to its conservation of discretized mass and energy.

\begin{theorem}
    The SI-DMD preserves the discretized mass, i.e. for any nonnegative integer $k$, 
    \begin{equation*}
        \label{eq:SI_mass_preserving}
        \left\|\bfx_{2k}\right\|_{2} = \left\|\bfx_0\right\|_{2}, \quad 
        \left\|\bfx_{2k+1}\right\|_{2} = \left\|\bfx_1\right\|_{2}.
    \end{equation*}
    In addition, it preserves the discretized quadratic energy functional 
    \(
        E(\bfx) = \bfx^* \bfA \bfx
    \)
    associated with the Hermitian operator $\bfA$, i.e., for any nonnegative integer $k$,
    \begin{equation*}
        E(\bfx_{2k}) = E(\bfx_0), \quad E(\bfx_{2k+1}) = E(\bfx_1).
    \end{equation*}
\end{theorem}
\begin{proof}
    The mass-conserving property \eqref{eq:SI_mass_preserving} follows from the unitarity of the evolution matrix 
    \begin{equation}
    \label{eq:SI_DMD_evolution_matrix}
        \bfL = \left(\eye + \imath \tau \bfA \right)^{-1}\left(\eye - \imath \tau \bfA \right).
    \end{equation}
    The energy-conserving property is established through the similar relation as \eqref{eq:CN_energy_preserving_iteration}.
\end{proof}

The SI-DMD method admits an error estimate analogous to that of CN-DMD in Subsection~\ref{subsec:error_analysis}. The only difference is that the estimates are carried out separately for even and odd time steps; all other arguments remain unchanged. For brevity, we omit the details.

\subsubsection{Model order reduction}
We consider the same model order reduction technique as \eqref{eq:CN_DMD_Hermite_matrix_reduced} and \eqref{eq:CN_DMD_H_reduced}, as SI-DMD shares a similar reason for truncation as given in Lemma~\ref{lemma:truncation_reason}. For consistency and convenience in writing, we omit the decomposition \eqref{eq:CN_DMD_Hermite_matrix_reduced}, \eqref{eq:CN_DMD_H_reduced} and \eqref{eq:CN_DMD_Hermite_matrix_reduced_eig}. It leads to a computationally friendly iteration, for any nonnegative integer $k$,
\begin{equation}
    \label{eq:SI_DMD_recurrence_reduced}
    \begin{aligned}
    \bfx_{2k} &= \bfU_{\widetilde{\bfA}} \, \bfD_{\widetilde{\bfH}}^k \, \bfU_{\widetilde{\bfA}}^* \bfx_0 + \left(\eye -\bfU_{\widetilde{\bfA}} \bfU_{\widetilde{\bfA}}^*\right) \bfx_0, \\
    \bfx_{2k+1} &= \bfU_{\widetilde{\bfA}} \, \bfD_{\widetilde{\bfH}}^k \,  \bfU_{\widetilde{\bfA}}^* \bfx_1 + \left(\eye -\bfU_{\widetilde{\bfA}} \bfU_{\widetilde{\bfA}}^*\right) \bfx_1,
    \end{aligned}
\end{equation}    
where $\bfU_{\widetilde{\bfA}} = \widetilde{\bfU}_{\bfX_1}\bfW_{\widetilde{\bfH}}$ is set as the reduced eigenvector matrix of $\widetilde{\bfA}$, and the diagonal matrix is defined as 
\begin{equation*}
    \bfD_{\widetilde{\bfH}} = \diag\Bigl(\frac{1-\imath \tau \lambda_j}{1+\imath \tau \lambda_j}\Bigr)_{j=1}^r.
\end{equation*}
Algorithm \ref{algo:SI_DMD} describes the overall procedure of SI-DMD for a single time prediction.

\begin{algorithm}[!tb]
\caption{The Semi-implicit DMD (single time prediction)}
\label{algo:SI_DMD}
\begin{algorithmic}[1]
    \State \textbf{Input:} Initial states $\bfx_0,\, \bfx_1$, data matrices $\bfX_1,\, \bfX_2 \in \C^{n \times m}$ constructed by~\eqref{eq:SI_DMD_data_matrix}, tolerance $\texttt{tol}$ for truncated SVD, time step $\tau$, target index $N$ (predict at time $N \tau$).
    \State $[\bfU_{\bfX_1}, \boldsymbol{\Sigma}_{\bfX_1}, \bfV_{\bfX_1}] \gets \svd(\bfX_1, \texttt{tol})$ \hfill \% truncated SVD with tolerance $\texttt{tol}$
    \State Construct the Hermitian $\widetilde{\bfH} \in \C^{r\times r}$ as in \eqref{eq:CN_DMD_H} and \eqref{eq:CN_DMD_H_reduced}
    \State $[\bfW_{\widetilde{\bfH}},\boldsymbol{\Lambda}_{\widetilde{\bfH}}] \gets \eig(\widetilde{\bfH})$
    \State $\bfU_{\widetilde{\bfA}} \gets \bfU_{\bfX_1} \bfW_{\widetilde{\bfH}}$ \hfill \% lift reduced eigenvectors to the full space
    \State We denote $\boldsymbol{\Lambda}_{\widetilde{\bfH}} = \diag(\lambda_1, \dots, \lambda_r)$, then \begin{equation*}
        \bfD_{\widetilde{\bfH}} \gets  \diag\Bigl(\frac{1-\imath \tau \lambda_j}{1+\imath \tau \lambda_j}\Bigr)_{j=1}^r
    \end{equation*}
    \State $\bfz_0 \gets \bfU_{\widetilde{\bfA}}^*
    \, \bfx_0$, \quad $\bfx_0^{\,\perp} \gets \bfx_{0} - \bfU_{\widetilde{\bfA}} \bfz_0$
    \State $\bfz_1 \gets \bfU_{\widetilde{\bfA}}^* \, \bfx^{1}$, \quad $\bfx_{1}^{\,\perp} \gets \bfx_{1} - \bfU_{\widetilde{\bfA}} \bfz_1$
    \If{$N$ is even} \Comment{$N=2M$}
        \State $M \gets N/2$; \quad $\bfz_N \gets \bfD_{\widetilde{\bfH}}^{\,M}\, \bfz_0$ \hfill \% elementwise power on the diagonal
        \State $\bfx_{N} \gets \bfU_{\widetilde{\bfA}} \bfz_N + \bfx_{0}^{\,\perp}$
    \Else \Comment{$N=2M+1$}
        \State $M \gets (N-1)/2$;\quad $\bfz_N \gets \bfD_{\widetilde{\bfH}}^{\,M} \bfz_1$ \hfill \% elementwise power on the diagonal
        \State $\bfx_{N} \gets \bfU_{\widetilde{\bfA}} \bfz_N + \bfx_{1}^{\,\perp}$
    \EndIf
    \State \textbf{Output:} Predicted state $\bfx_N$ at time $N \tau$.
\end{algorithmic}
\end{algorithm}

\subsection{Parallelism of the methods}

In this subsection, we give some practical suggestions for parallelism according to the well-formed iteration formulas~\eqref{eq:CN_DMD_iteration} and~\eqref{eq:SI_DMD_recurrence_reduced}.

\paragraph{Parallel prediction at all intermediate times}
First, We notice that the prediction at different time can be computed independently. Let $\bfd=\diag(\bfD_{\widetilde{\bfH}})\in\C^{r}$ be the diagonal of $\bfD_{\widetilde{\bfH}}$. Let $\bfz_0 = \bfU_{\widetilde{\bfA}}^* \bfx_0 \in \C^r$ and $\bfx_0^{\perp} = \bfx_{0} - \bfU_{\widetilde{\bfA}} \bfz_0 \in \C^r$. 
For each $k \in\{1,\dots,N\}$,
\begin{equation*}
  \bfx_{k} = \bfU_{\widetilde{\bfA}} \,\big(\bfd^k\odot \bfz_0\big) + \bfx_0^{\perp},
\end{equation*}
where $\odot$ denotes the Hadamard (elementwise) product, and $\bfd=[d_1,\dots,d_r]\trans \in \C^r$, where $\bfd^k$ denotes the elementwise $k$-th power of $\bfd$. Since each $k$ uses only elementwise powers of $\bfd$ and a single lift by $\bfU_{\widetilde{\bfA}}$, $\{\bfx_k\}_{k=0}^N$ can be computed efficiently in parallel over~$k$.

\paragraph{Block prediction at all intermediate times}
Another efficient approach is to stack the reduced states $\{\bfz_k\}_{k=1}^N$, where $\bfz_k = \bfd^{k}\odot \bfz_0$, into a matrix 
\begin{equation*}
   \bfZ=\bigl[\bfz_1, \cdots, \bfz_N \bigr]\in\C^{r \times N}.
\end{equation*}
Then all lifted states can be obtained with few matrix-matrix multiplications:
\begin{equation*}
   \bigl[\bfx_1, \cdots, \bfx_N\bigr]
   = \bfU_{\widetilde{\bfA}} \bfZ + \bigl[\bfx_0^\perp,\cdots,\bfx_0^\perp \bigr] \in \C^{n\times N}.
\end{equation*}
This block formulation avoids explicit parallelization over $k$ and exploits highly optimized BLAS-3 matrix-matrix multiplication for throughput.
\medskip


\begin{remark} For a large $k$, use $\bfd^k = \exp\!\big(k\log \bfd\big)$ (elementwise) to improve stability and efficiency.
\end{remark}

\subsection{Comparison of the complexity}

We assume that the data matrices \(\bfX_1\), \(\bfX_2 \in \C^{n \times m}\) with \(m \ll n\), and the numerical rank\footnote{The numerical rank is defined as the number of the truncated singular values.} of \(\bfX_1\) is \(r\). The length of the prediction trajectory is set as $N$.

We compare various DMD-type algorithms in terms of their computational costs for prediction, as well as the associated storage costs of the evolution matrix. Table \ref{tab:predict_time_comparation} shows the computational and storage costs of different algorithms. The dominant computational cost \(\mathcal{O}(n m^2)\) in Crank-Nicolson DMD (CN-DMD) and semi-implicit DMD (SI-DMD), which arises from the full SVD of \(\bfX_1\), can be reduced to \(\mathcal{O}(n m r)\) by employing a truncated SVD. As shown in Table \ref{tab:predict_time_comparation}, the costs of CN/SI-DMD and classical DMD are of the same order of magnitude, but our methods preserve mass conservation and exhibit greater robustness to noise.

\begin{table}[H]
\footnotesize
\caption{Prediction-time and storage complexity: CN/SI-DMD vs. classical DMD vs. direct piDMD.}
\label{tab:predict_time_comparation}
\begin{center}
\begin{tabular}{lcccc}
\toprule
 & CN/SI-DMD & classical DMD & direct piDMD \\
\midrule
single target $\bfx_N$ 
& $\mathcal{O}(nmr)$ 
& $\mathcal{O}(nmr)$ 
& $\mathcal{O}(n^3 + n^2N)$ \\
all targets $\bfx_{1:N}$ 
& $\mathcal{O}(nmr + n r N)$ 
& $\mathcal{O}(nmr + n r N)$ 
& $\mathcal{O}(n^3 + n^2N)$ \\
storage
& $\mathcal{O}(nr)$ 
& $\mathcal{O}(nr)$
& $\mathcal{O}(n^2)$ \\
\bottomrule
\end{tabular}
\end{center}
\end{table}


\section{Numerical examples}
\label{sec:numerical_examples}
In this section, we illustrate through several numerical experiments the model order reduction properties, computational efficiency, and robustness to noise of the proposed Crank--Nicolson DMD and semi-implicit DMD. To examine the intrinsic properties of the models and to ensure a fair comparison, we compare our methods with classical DMD and piDMD directly on the wave function data, without any preprocessing. The training data in all experiments are generated using the Strang splitting spectral method~\cite{baoTimeSplittingSpectralApproximations2002}.


\paragraph{Computational environment}
The algorithms were implemented and tested in MATLAB R2023b. The reported runtimes are measured using MATLAB's \texttt{tic}/\texttt{toc} functions.

\subsection{Forward wave propagation} In this experiment, we show that our algorithms can simulate the forward propagation of the wave in both space and time. 
\label{subsec:propagation}
We consider the semiclassical \Sch\ equation~\eqref{eq:schrodinger_equation} with the initial condition of a WKB form~\cite{baoTimeSplittingSpectralApproximations2002},
\begin{equation}
\label{eq:propagation_forward_initial_condition}
u^{\varepsilon}(x,0)
    = u_0^{\varepsilon}(x)
    = \sqrt{n_0(x)}\, e^{\imath S_0(x)/\varepsilon},
\end{equation}
where
\begin{equation}
\label{eq:propagation_forward_initial_condition_explicit}
n_0(x) = \bigl(e^{-25(x - \frac{a+b}{2})^2}\bigr)^2,
\qquad
S_0(x) = -\frac{1}{50}(x-a)(x-b),
\qquad x \in \R.
\end{equation}
We further choose the spatial domain \([a,b] = [0,2]\), the constant potential \(V(x)\equiv 10\), and the scaled Planck constant \(\varepsilon = 10^{-2}\). 


We set both the time-step size \(\tau\) and the spatial mesh size \(h\) to be \(10^{-2}\). Then we construct the data matrix $\bfX = [\bfu_1, \bfu_2, \dots, \bfu_{99}, \bfu_{100}] \in \C^{200 \times 100}$, where $\bfu_k:= [u^\varepsilon(x_1, t_k), u^\varepsilon(x_2, t_k), \cdots, u^\varepsilon(x_{200}, t_k)]\trans$, with $x_j = jh$, $t_k = k\tau$ ($j=1, 2, \cdots, 200$, $k\geq 0$), denoting the spatial and temporal grids, respectively. 
The data are generated using the time-splitting method with the same temporal and spatial mesh sizes $\tau_e = h_e = 10^{-2}$.  We evaluate the performance of different methods over a time horizon eight times longer than that of the given data, with the results shown in Figure~\ref{fig:forward_propagation}. 

Figure~\ref{fig:forward_propagation} compares the forward propagation obtained by several DMD-type methods with the true propagation, where all displayed values correspond to the absolute values of the complex-valued wave function. All methods can simulate the spatio-temporal forward propagation in general, but the three methods with model order reduction outperform piDMD, as piDMD suffers from significant prediction errors. Analysis of the toy model in Remark~\ref{rmk:toy_model} suggests that clustered singular values lead to inaccurate singular vector matrices $\boldsymbol{\Phi}_r$ and their orthogonal complement $\boldsymbol{\Phi}_r^{\bot}$. Consequently, the additional term $\boldsymbol{\Phi}_r^{\bot}(\boldsymbol{\Phi}_r^{\bot})^*$, which renders $\bfL^{\rm piDMD}$ unitary, introduces extra error relative to the classical DMD. 

\begin{figure}[htbp]
\centering
\begin{subfigure}[t]{0.32\textwidth}
    \includegraphics[width=\linewidth]{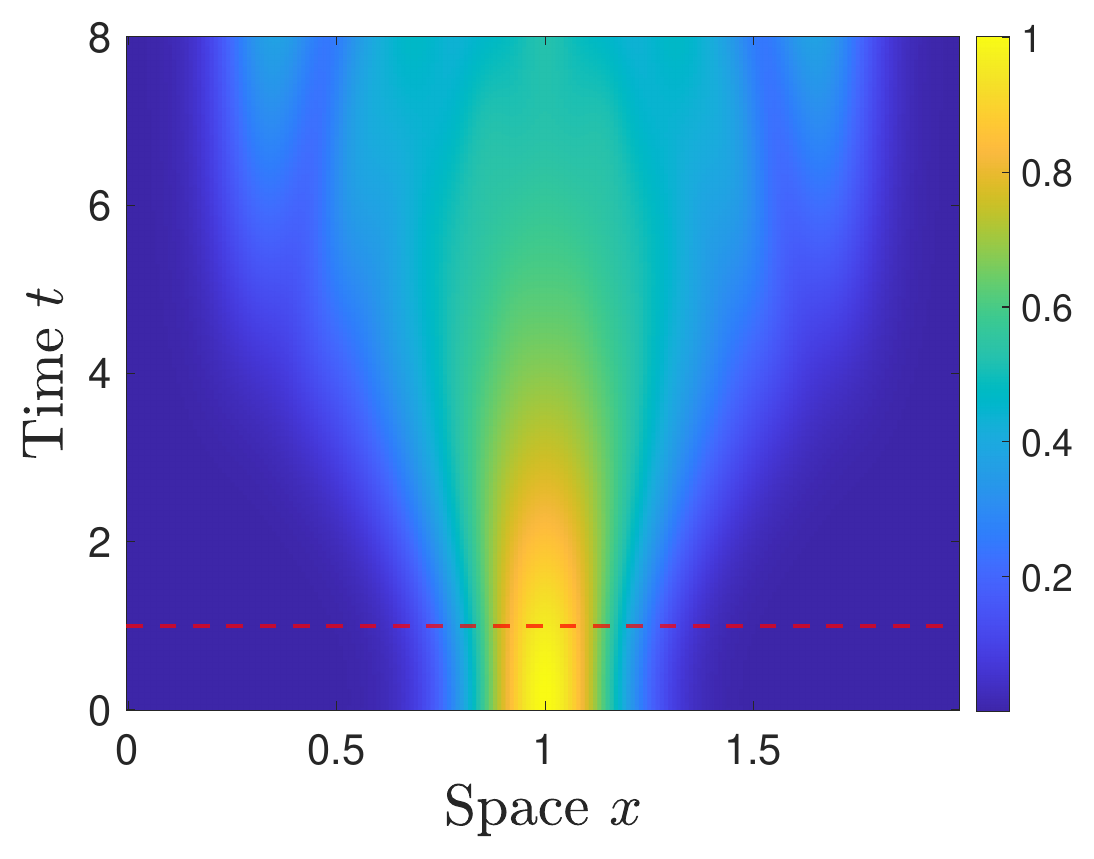}
    \subcaption{CN-DMD}
\end{subfigure}
\begin{subfigure}[t]{0.32\textwidth}
    \includegraphics[width=\linewidth]{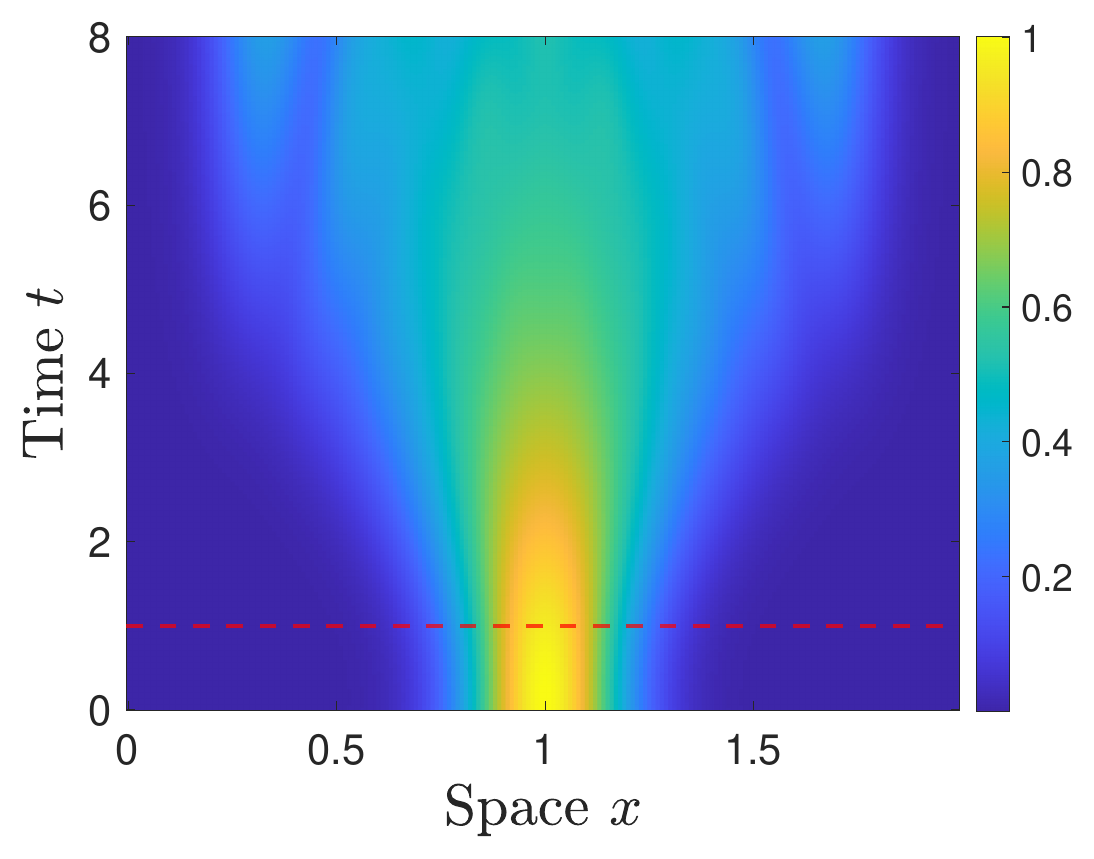}
    \subcaption{SI-DMD}
\end{subfigure}
\begin{subfigure}[t]{0.32\textwidth}
    \includegraphics[width=\linewidth]{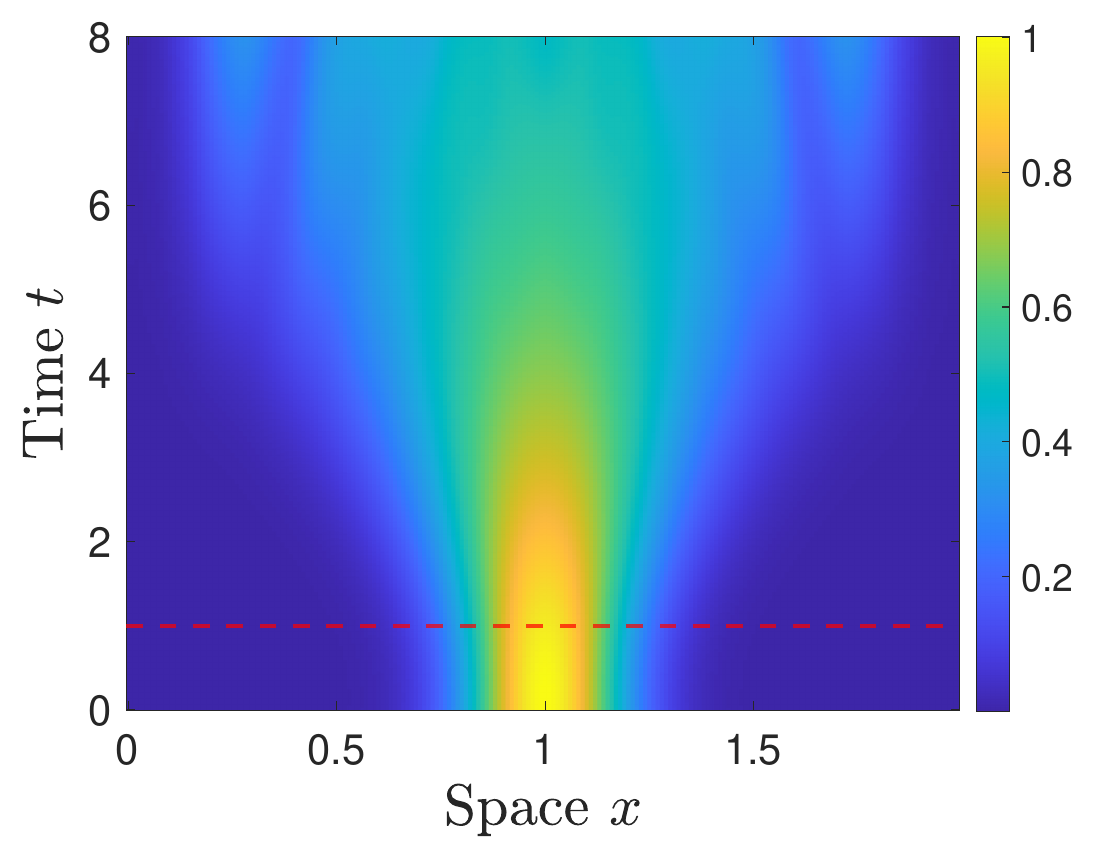}
    \subcaption{Classical DMD}
\end{subfigure}
\begin{subfigure}[t]{0.32\textwidth}
    \includegraphics[width=\linewidth]{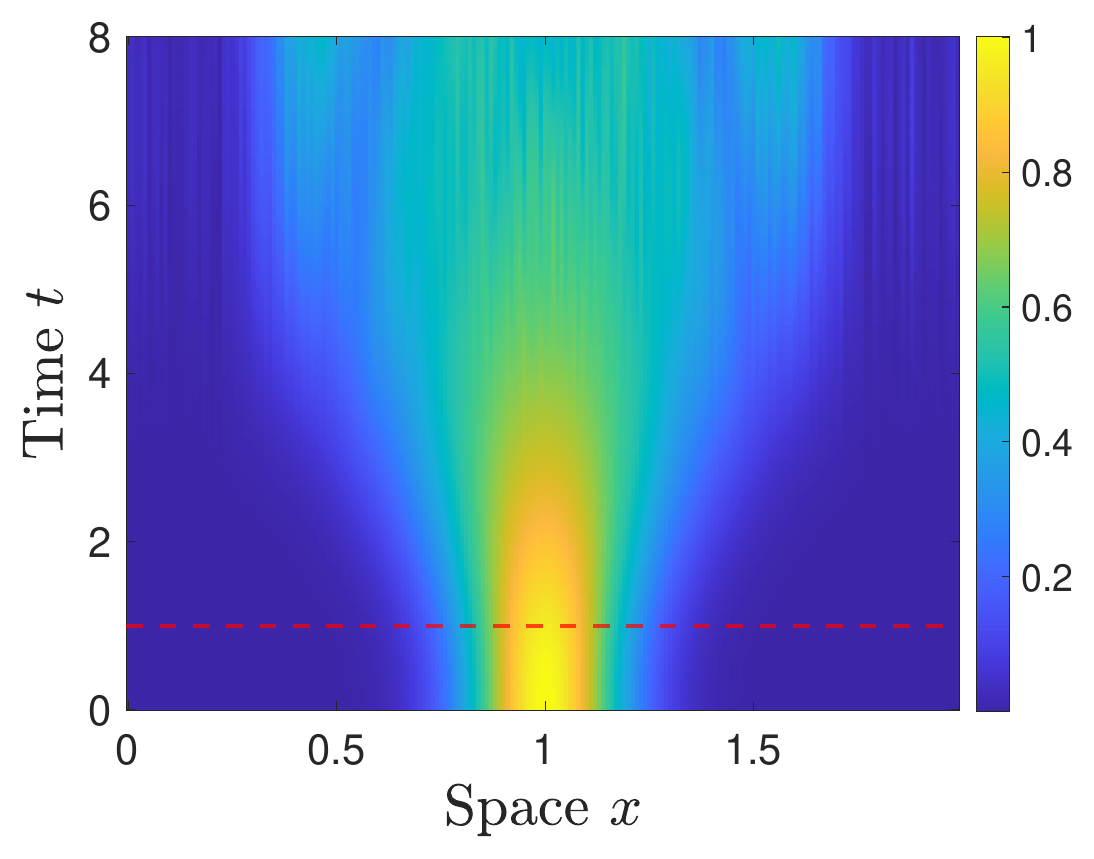}
    \subcaption{piDMD}
\end{subfigure}
\begin{subfigure}[t]{0.32\textwidth}
    \includegraphics[width=\linewidth]{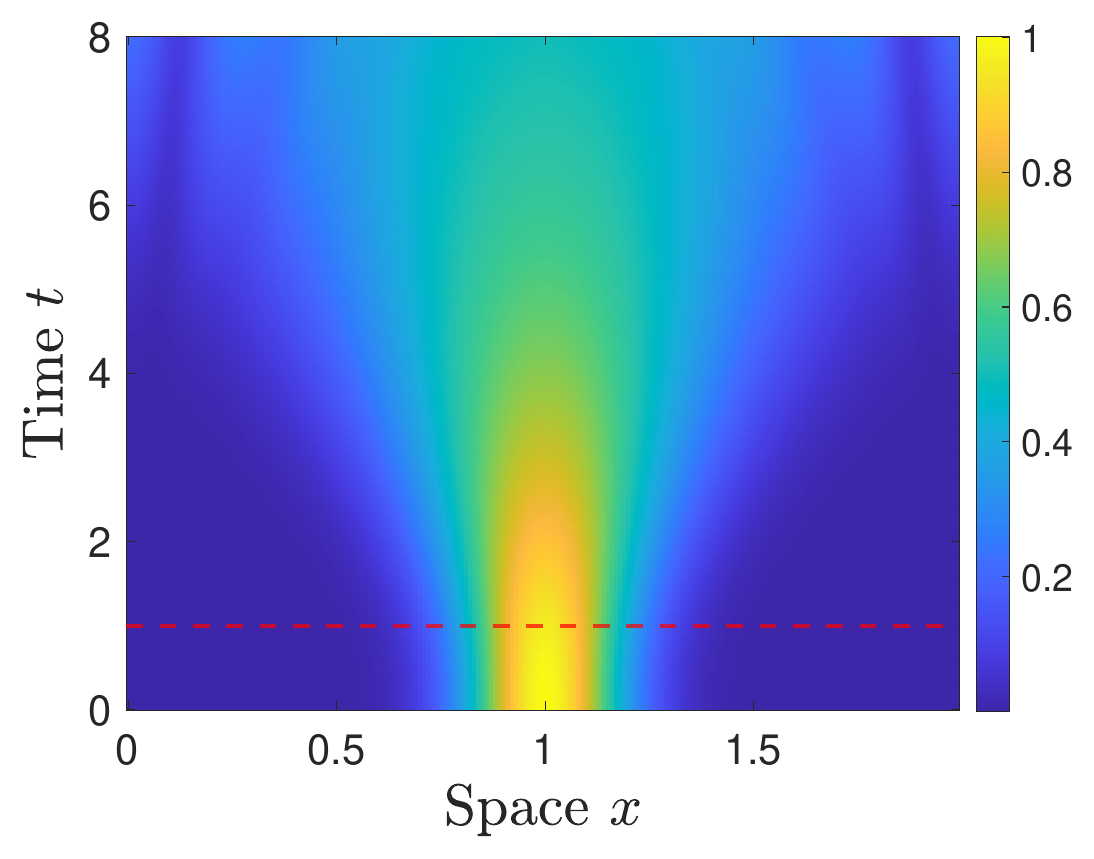}
    \subcaption{True propagation}
\end{subfigure}
\caption{Forward wave propagation obtained by different DMD-type methods: (a) CN-DMD; (b) SI-DMD; (c) Classical DMD; (d) piDMD; compared with (e) True propagation. The red dotted line indicates the final time of the given data.}
\label{fig:forward_propagation}
\end{figure}

\subsection{Mass and energy conservation under noise}
\label{subsec:noise}
In this experiment, we investigate the accuracy of reconstruction and prediction under noisy data, as well as the mass and energy conservation properties of different methods in predictive tasks.

We consider the semiclassical \Sch\ equation~\eqref{eq:schrodinger_equation} with an initial condition of the WKB type~\eqref{eq:propagation_forward_initial_condition} with 
$n_0(x)$ and $S_0(x)$ defined in \eqref{eq:propagation_forward_initial_condition_explicit}. We further choose the spatial domain \([a,b] = [0,1]\), the harmonic oscillator \(V(x) = 10x^2\), and the scaled Planck constant \(\varepsilon = 10^{-2}\).

We set the time-step size to \(\tau = 10^{-2}\) and the spatial mesh size to \(h = 4 \times 10^{-3} \). And we construct the data matrix
\(
\bfX = [\bfu_0, \bfu_1, \bfu_2, \ldots, \bfu_{79}] \in \C^{250 \times 80},
\)
where $\bfu_k$ denotes the wave function evaluated on an equally spaced spatial grid at time $k\tau$. The data are generated using the time-splitting method with temporal and spatial mesh sizes \(\tau_e = h_e = 10^{-3}\), and then downsampled onto the coarser grid used in the data matrix. The training data are then generated by adding complex Gaussian noise to \(\bfX\) elementwisely, given by
\[
    \eta = \frac{\sigma}{\sqrt{2}}\bigl(\eta_1 + \imath \eta_2\bigr) \stackrel{i.i.d}{\sim}  \mathcal{CN}(0,\sigma^2),
\]
where \(\eta_1\) and \(\eta_2\) are independent standard normal random variables, and the noise level is controlled by \(\sigma\).


We measure the accuracy by computing the relative error between the predicted and true wave functions,
\begin{equation*}
    \err_k = \bigl\|\bfx^{\text{pred}}_k - \bfx^{\text{true}}_k \bigr\|_{2} \,/\, \bigl\|\bfx^{\text{true}}_k\bigr\|_{2}
\end{equation*}
at time $t_k = k\tau$. In addition, we quantify mass and energy conservation by measuring the relative variation
\begin{align}
\mv_k &= \bigl| \| \bfx^{\text{pred}}_k \|_2 - \| \bfx^{\text{pred}}_0 \|_2 \bigr|
        \,/\, \| \bfx^{\text{pred}}_0 \|_2,
\label{eq:rel_mass_variation}
\\
\ev_k &= \bigl| |E(\bfx^{\text{pred}}_k)| - |E(\bfx^{\text{pred}}_0)| \bigr|
        \,/\, |E(\bfx^{\text{pred}}_0)|,
\label{eq:rel_energy_variation}
\end{align}
respectively, where $E(\cdot)$ denotes the discretized energy defined in~\eqref{eq:discretized_energy}. We test several DMD-type methods and record their training and prediction errors under different noise levels in Figure~\ref{fig:prediction_under_different_level_noise}, as well as their mass and energy conservation properties in Figure~\ref{fig:mass_energy_conservation}.

Figure~\ref{fig:prediction_under_different_level_noise} shows that the three methods incorporating physical constraints are significantly more robust than classical DMD under noisy conditions. Classical DMD tends to overfit the noisy training data, leading to prediction errors that grow much more rapidly than those of the structure-preserving methods in high-noise regimes.

Moreover, our proposed CN-DMD and SI-DMD exhibit noise robustness comparable to that of piDMD. The robustness of piDMD can be attributed to its equivalence to a total least square formulation with a unitary constraint, which is known to improve stability under noisy measurements. It demonstrates the noise robustness of our methods.

Figure~\ref{fig:mass_energy_conservation} demonstrates that, under noisy data, classical DMD fails to preserve mass, whereas our methods CN-DMD and SI-DMD maintain both discretized mass and energy with high accuracy.


\begin{figure}[tb]
\centering
\begin{subfigure}[t]{0.48\textwidth}
    \includegraphics[width=\textwidth]{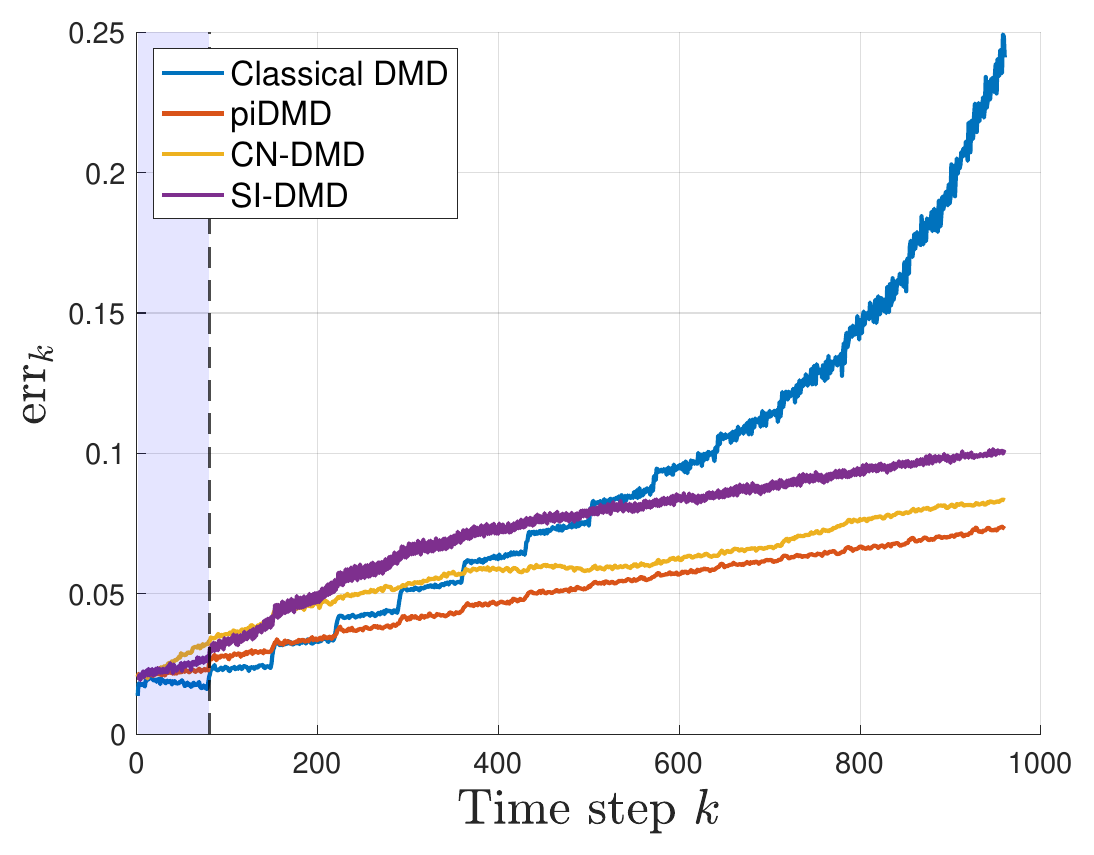}
    \subcaption{$\sigma = 10^{-2}$}
\end{subfigure}
\begin{subfigure}[t]{0.48\textwidth}
    \includegraphics[width=\textwidth]{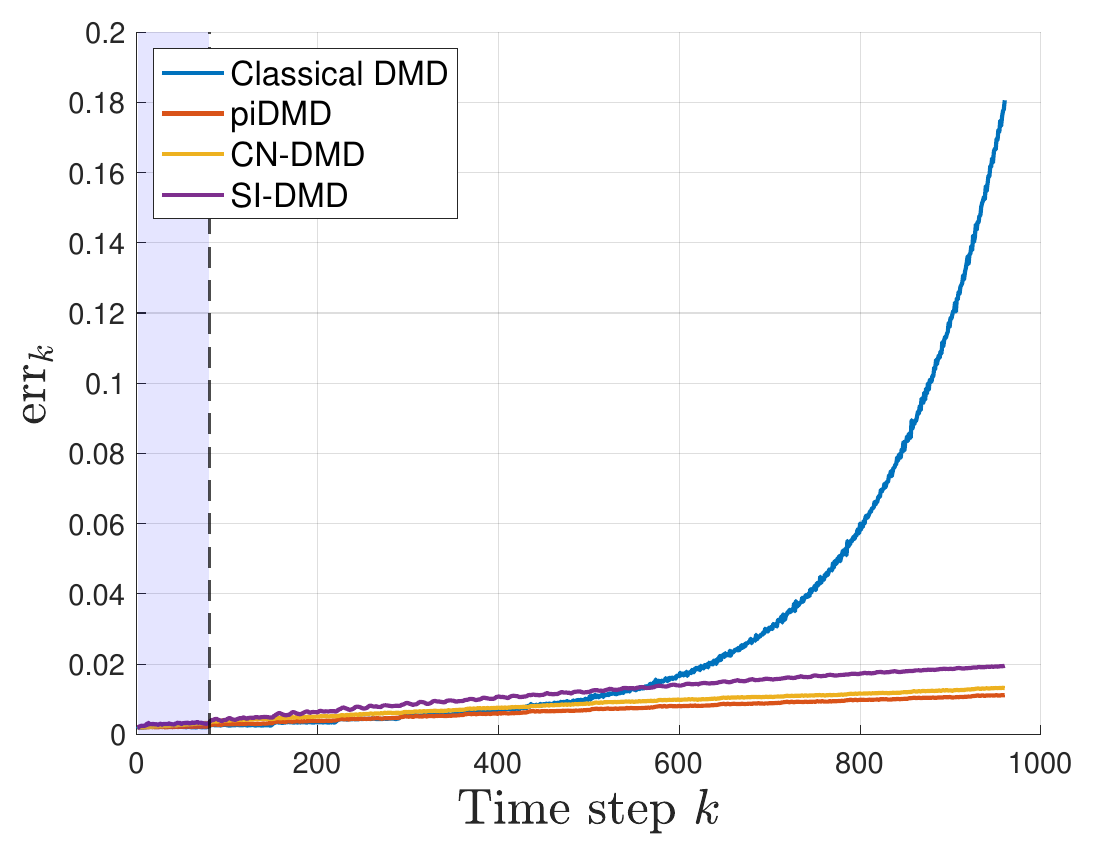}
    \subcaption{$\sigma = 10^{-3}$}
\end{subfigure}

\begin{subfigure}[t]{0.48\textwidth}
    \includegraphics[width=\textwidth]{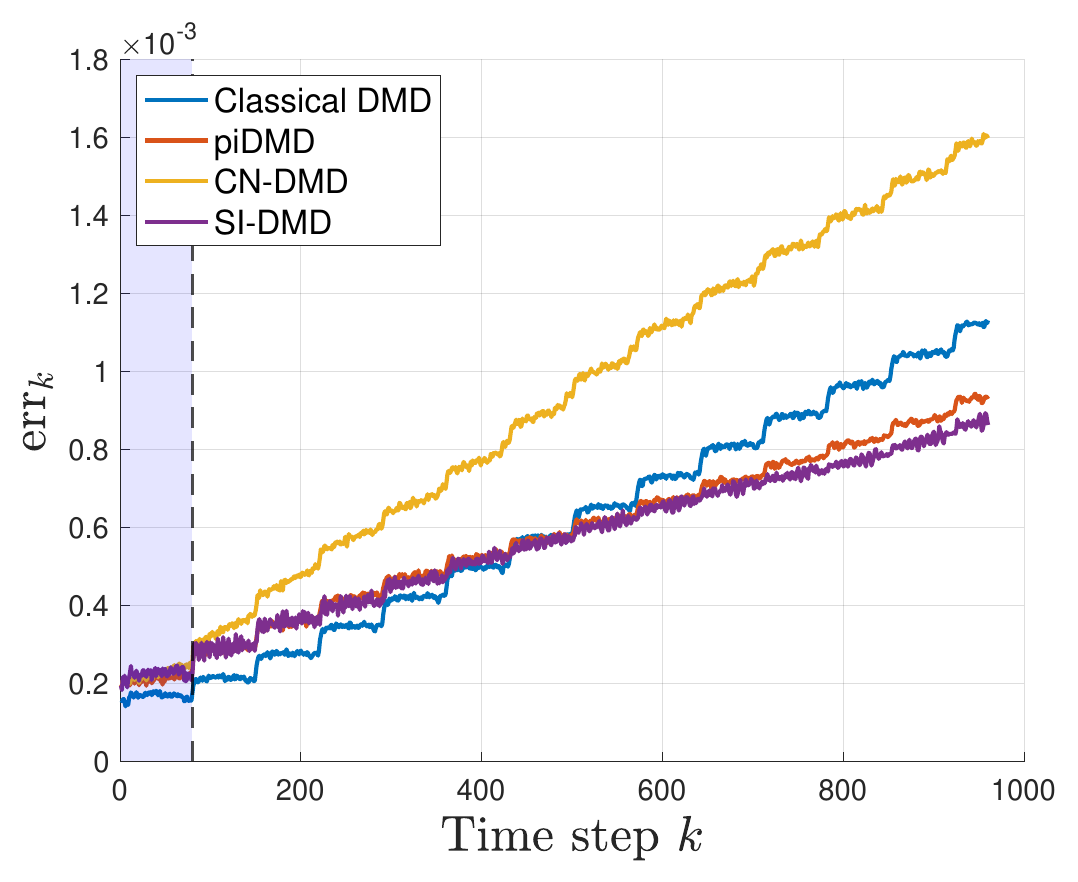}
    \subcaption{$\sigma = 10^{-4}$}
\end{subfigure}
\begin{subfigure}[t]{0.48\textwidth}
    \includegraphics[width=\textwidth]{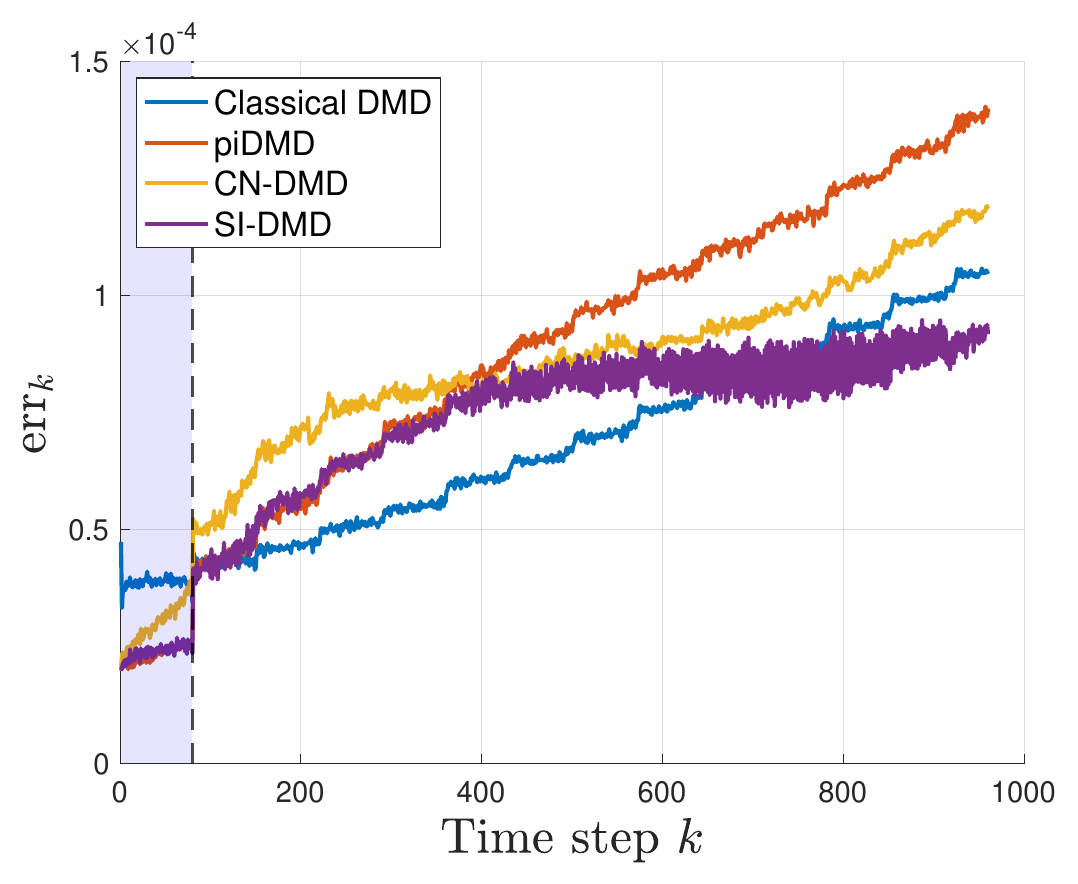}
    \subcaption{$\sigma = 10^{-5}$}
\end{subfigure}
\caption{Training and prediction error under different levels of noise (a) $\sigma = 10^{-2}$; (b) $\sigma = 10^{-3}$; (c) $\sigma = 10^{-4}$; (d) $\sigma = 10^{-5}$.
}
\label{fig:prediction_under_different_level_noise}
\end{figure}

\begin{figure}[tb]
\centering
\begin{subfigure}[t]{0.48\textwidth}
    \includegraphics[width=\textwidth]{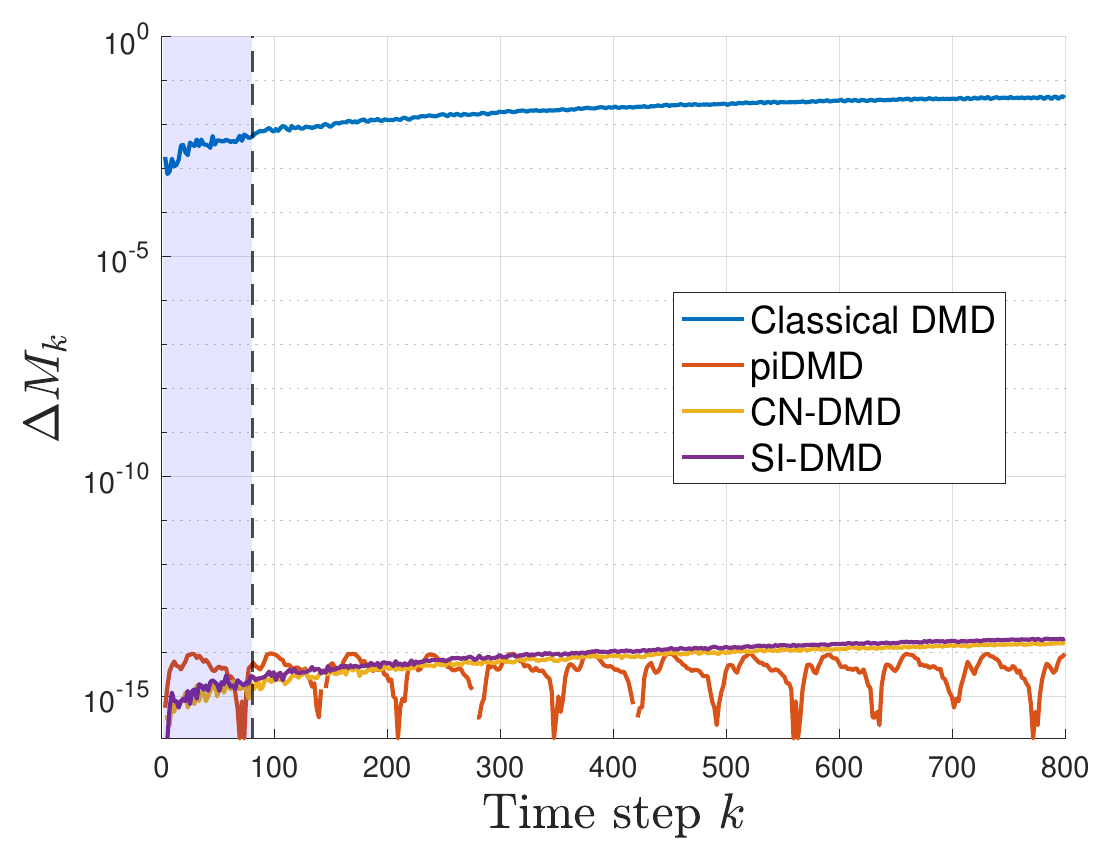}
\end{subfigure}
\begin{subfigure}[t]{0.48\textwidth}
    \includegraphics[width=\linewidth]{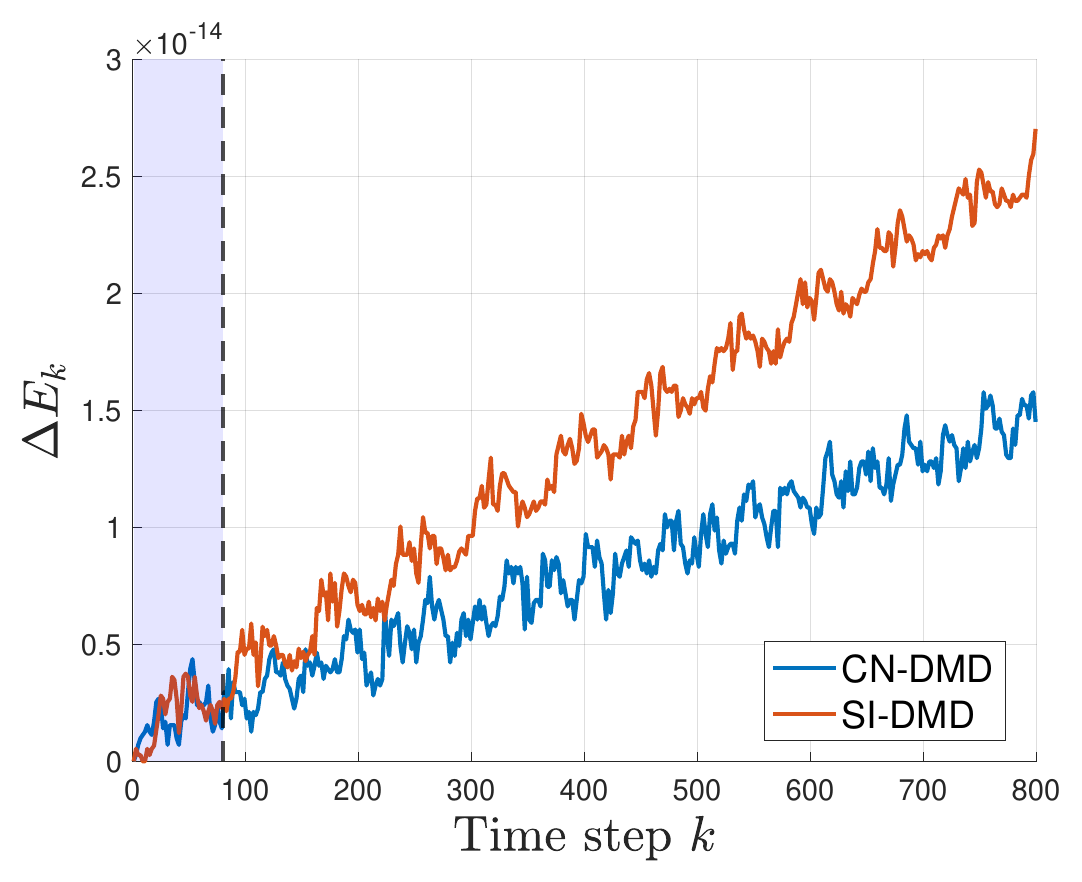}
\end{subfigure}
\caption{
(a) Mass variation. (b) Energy variation. The dotted line indicates the final time step of the train data.
}
\label{fig:mass_energy_conservation}
\end{figure}

\subsection{Comparison of computational costs} 
\label{subsec:highdim}
In this experiment, we mainly focus on the computational efficiency of those methods. We consider the semiclassical \Sch\ equation~\eqref{eq:schrodinger_equation} with an initial condition of the WKB type~\eqref{eq:propagation_forward_initial_condition}
where
\begin{equation*}
n_0(x) = \bigl(e^{-25(x - \frac{a+b}{2})^2}\bigr)^2, \quad 
S_0(x)=-\frac{1}{5} \ln \bigl(e^{5(x - \frac{a+b}{2})}+e^{-5(x - \frac{a+b}{2})}\bigr), 
\quad x \in \R.
\end{equation*}
We choose the spatial domain \([a, b] = [0, 10]\), with a constant potential \(V(x) \equiv 10\) and the scaled Planck constant \(\varepsilon = 10^{-2}\).

We set both the time-step size $\tau$ and the spatial mesh size $h$ to $10^{-3}$. And we construct the train data as the matrix $\bfX = [\bfu_0, \bfu_1, \dots, \bfu_{48}, \bfu_{49}] \in \mathbb{C}^{10000 \times 50}$, where $\bfu_k$ denotes the wave function evaluated on an equally spaced spatial grid at time $k\tau$. The data are generated using the time-splitting method with the same temporal and spatial mesh sizes $\tau_e = h_e = 10^{-3}$. Our task is to evaluate the performance
of different methods over a time horizon that is eight times longer than the given data window.

We use the relative prediction error 
\begin{equation}
\label{eq:relative_err}
    e_{\mathrm{rel}} = \bigl\|\bfX^{\mathrm{pred}} - \bfX^{\mathrm{true}} \bigr\|_{\mathrm{F}}\, /\,
         \bigl\|\bfX^{\mathrm{true}}\bigr\|_{\mathrm{F}},
\end{equation}
to simultaneously quantify both fitting and prediction accuracy. Here, \(\bfX^{\mathrm{pred}},\, \bfX^{\mathrm{true}} \in \C^{10000 \times 400}\) collect the data at time points \(t_k\) for \(0 \le k \le 399\), including both the fitting interval and the prediction horizon. In addition, we evaluate mass conservation by measuring the relative mass $\mv_{399}$~\eqref{eq:mass} at the final prediction time $t_{399}$.

We report the runtime of these methods in Table~\ref{tab:runtime}. The runtime consists of two components: the time required to execute the DMD-type algorithm and the time required to generate predictions up to the prescribed final time. The results show that, compared with classical DMD, our methods achieve significantly better mass preservation while maintaining comparable computational cost and prediction accuracy. On the other hand, compared with direct piDMD, our methods attain higher accuracy with substantially reduced runtime, without requiring any data-preprocessing techniques. 

\begin{table}[H]
\footnotesize
\centering
\caption{Runtime and accuracy comparison of different DMD-type algorithms.}
\label{tab:runtime}
\begin{tabular}{l c c c}
\hline
Method    & Runtime (seconds) & $e_{\text{rel}}$       & $\mv_{399}$  \\
\hline
CN-DMD    & 0.72              & 6.54\ee{-2}  & 2.46\ee{-14}           \\
SI-DMD    & 0.85              & 6.98\ee{-2}  & 8.98\ee{-16}           \\
Classical DMD & 0.81              & 4.57\ee{-2}  & 4.75\ee{-4} \\
piDMD    & 805.07      & 1.67\ee{-1}  & 2.45\ee{-13}\\
\hline
\end{tabular}
\end{table}

\subsection{Performance in the semiclassical regime}
In this experiment, we compare the performance of our proposed CN-DMD and SI-DMD methods under different values of the scaled Planck constant $\varepsilon$. We consider the \Sch\ equation with the WKB initial condition \eqref{eq:propagation_forward_initial_condition}, $\eqref{eq:propagation_forward_initial_condition_explicit}$ with the spatial domain \([a,b] = [0, 1]\) and potential \(V(x) = 10x^2\) for all different $\varepsilon$.

For comparison purpose, we set the time step size \(\tau\) to \(10^{-2}\) and the spatial mesh size \(h\) to  \(10^{-3}\) for all considered $\varepsilon$. We further construct the data matrix \(\bfX^{\varepsilon} = [\bfu_0^{\varepsilon}, \bfu_1^{\varepsilon}, \dots, \bfu_{m-1}^{\varepsilon}] \in \C^{1000 \times m}\), where \(\bfu_k^{\varepsilon}\) denotes the wave function evaluated on an equally spaced spatial grid over \([a, b]\) with spatial mesh size \(h\) at time \(k\tau\) under the scaled Planck constant $\varepsilon$, and \(m\) denotes the horizon of the train data $\bfX^{\varepsilon}$. The data are generated using the time-splitting method with the temporal and spatial mesh sizes $\tau_e = h_e = 10^{-4}$ and are then downsampled onto the corresponding grids used in the data matrix.

Given the training data \(\bfX^{\varepsilon}\), we compare the performance of different methods in both training and long-time extrapolation. Specifically, the models are trained on data over a time horizon of length \(m\) and then used to predict dynamics over a horizon nine times longer than the training interval. We vary the number of snapshots and the scaled Planck constant over the ranges
\begin{equation*}
    m \in \{10,20,40,60,80\}, \qquad 
    \varepsilon \in \{2^{-0},2^{-2},2^{-4},2^{-6}\},
\end{equation*}
respectively. The prediction accuracy is evaluated using the relative Frobenius-norm error defined in~\eqref{eq:relative_err}. The results are reported in Table~\ref{tab:various_planck_accuracy}.

\begin{table}[htb]
\footnotesize
\centering
\caption{training and prediction errors (relative Frobenius norm) of CN-DMD and SI-DMD for different snapshot numbers \(m\) and scaled Planck constants \(\varepsilon\).}
\label{tab:various_planck_accuracy}
\setlength{\tabcolsep}{8pt}
\renewcommand{\arraystretch}{1.3}
\begin{tabular}{cc cccc}
\toprule
\multicolumn{2}{c}{Snapshots $m$} & \multicolumn{4}{c}{$\varepsilon$} \\
\cmidrule(lr){3-6}
 &  & $2^{-0}$ & $2^{-2}$ & $2^{-4}$ & $2^{-6}$ \\
\midrule
\multirow{2}{*}{10}  & CN-DMD & $9.37\ee{-4}$ & $1.17\ee{-1}$ & $2.81\ee{-2}$ & $8.54\ee{-1}$ \\
                     & SI-DMD & $4.32\ee{-4}$ & $2.37\ee{-1}$ & $6.00\ee{-2}$ & $4.96\ee{-1}$ \\
\midrule
\multirow{2}{*}{20}  & CN-DMD & $3.55\ee{-4}$ & $1.11\ee{-2}$ & $4.34\ee{-3}$ & $3.60\ee{-1}$ \\
                     & SI-DMD & $4.33\ee{-4}$ & $5.00\ee{-4}$ & $6.40\ee{-3}$ & $5.01\ee{-1}$ \\
\midrule
\multirow{2}{*}{40}  & CN-DMD & $3.35\ee{-4}$ & $1.47\ee{-4}$ & $2.58\ee{-4}$ & $1.27\ee{-1}$ \\
                     & SI-DMD & $4.59\ee{-4}$ & $1.13\ee{-4}$ & $2.76\ee{-4}$ & $1.61\ee{-1}$ \\
\midrule
\multirow{2}{*}{60}  & CN-DMD & $5.86\ee{-5}$ & $1.07\ee{-4}$ & $1.64\ee{-4}$ & $2.72\ee{-4}$ \\
                     & SI-DMD & $1.49\ee{-4}$ & $1.10\ee{-4}$ & $9.60\ee{-5}$ & $3.12\ee{-2}$ \\
\midrule
\multirow{2}{*}{80}  & CN-DMD & $5.57\ee{-5}$ & $6.57\ee{-5}$ & $5.24\ee{-5}$ & $3.23\ee{-5}$ \\
                     & SI-DMD & $4.41\ee{-5}$ & $1.03\ee{-4}$ & $8.68\ee{-5}$ & $1.13\ee{-4}$ \\
\bottomrule
\end{tabular}
\end{table}

Table~\ref{tab:various_planck_accuracy} shows that our two methods, CN-DMD and SI-DMD, exhibit comparable performance across different choices of the number of snapshots \(m\) and the reduced Planck constant \(\varepsilon\). As \(\varepsilon\) decreases, a larger number of snapshots $m$ is required to accurately capture the increasingly oscillatory dynamics. Nevertheless, given sufficiently long training data, both methods achieve satisfactory accuracy in reconstruction and long-time prediction of the dynamics. These results indicate that our methods are well suited as surrogate models to accelerate traditional numerical solvers and therefore reduce overall computational costs.

\subsection{The nonlinear \Sch\ equation}
In this experiment, we assess the performance of our methods on the weak \(O(\varepsilon)\) defocusing Gross-Pitaevskii dynamics in a harmonic trap~\cite{baoUniformErrorEstimates2012},
\begin{equation}
\label{eq:nonlinear_schrodinger_equation}
    \varepsilon u_t^{\varepsilon} 
    - \imath \frac{\varepsilon^2}{2} u_{xx}^{\varepsilon}
    + \imath V(x) u^{\varepsilon}
    + \varepsilon |u^{\varepsilon}|^2 u^{\varepsilon} = 0,
    \qquad a < x < b,\; t > 0,
\end{equation}
under periodic boundary conditions, with initial conditions given in
\eqref{eq:propagation_forward_initial_condition}, where
\begin{equation}
\label{eq:nonlinear_initial_condition_explicit}
n_0(x) = \bigl(e^{-50(x - \frac{a+b}{2})^2}\bigr)^2,
\qquad
S_0(x) = -\frac{1}{50}(x-a)(x-b),
\qquad x \in \R.
\end{equation}
The total mass and energy of this system are also onserved~\cite{baoNumericalStudyTimesplitting2003}. In the numerical tests, we choose the spatial domain \([a,b] = [-3,3]\), the scaled Planck constant \(\varepsilon = 10^{-2}\), and the harmonic potential \(V(x) = 10x^2\). Thus we can apply CN-DMD and SI-DMD to predict dynamics of the system effectively.

We set the time-step size to be \(\tau = 3 \times 10^{-2}\) and the spatial mesh size to be \(h = 6 \times 10^{-3}\). We further construct the data matrix $\bfX = [\bfu_0, \bfu_1, \dots, \bfu_{48}, \bfu_{49}] \in \mathbb{C}^{1000 \times 50}$, where $\bfu_k$ denotes the wave function evaluated on an equally spaced spatial grid over $[a, b]$ with spatial mesh size $h$ at time $t_k = k\tau$. The data are generated using the time-splitting method with the time-step size \(\tau_e = 10^{-2}\) and the spatial mesh size \(h_e = 6 \times 10^{-3}\), and are then downsampled onto the coarser grid used in the data matrix. Under this setting, we compare the performance of several DMD-type methods and present the results in Figure~\ref{fig:nonlinear_error}.

Figure~\ref{fig:nonlinear_error}(a) shows that methods incorporating physical conservation laws outperform the classical DMD for this nonlinear equation. The classical DMD suffers from severe underfitting, resulting in noticeably poorer predictive accuracy.
Figure~\ref{fig:nonlinear_error}(b) illustrates that, with the inclusion of time-delay embedding~\cite{kambTimeDelayObservablesKoopman2020}, the corresponding high-order CN-DMD and SI-DMD variants achieve higher accuracy and slower error accumulation than the other methods considered.

\begin{figure}[!tbp]
\centering
\begin{subfigure}[t]{0.48\textwidth}
    \includegraphics[width=\linewidth]{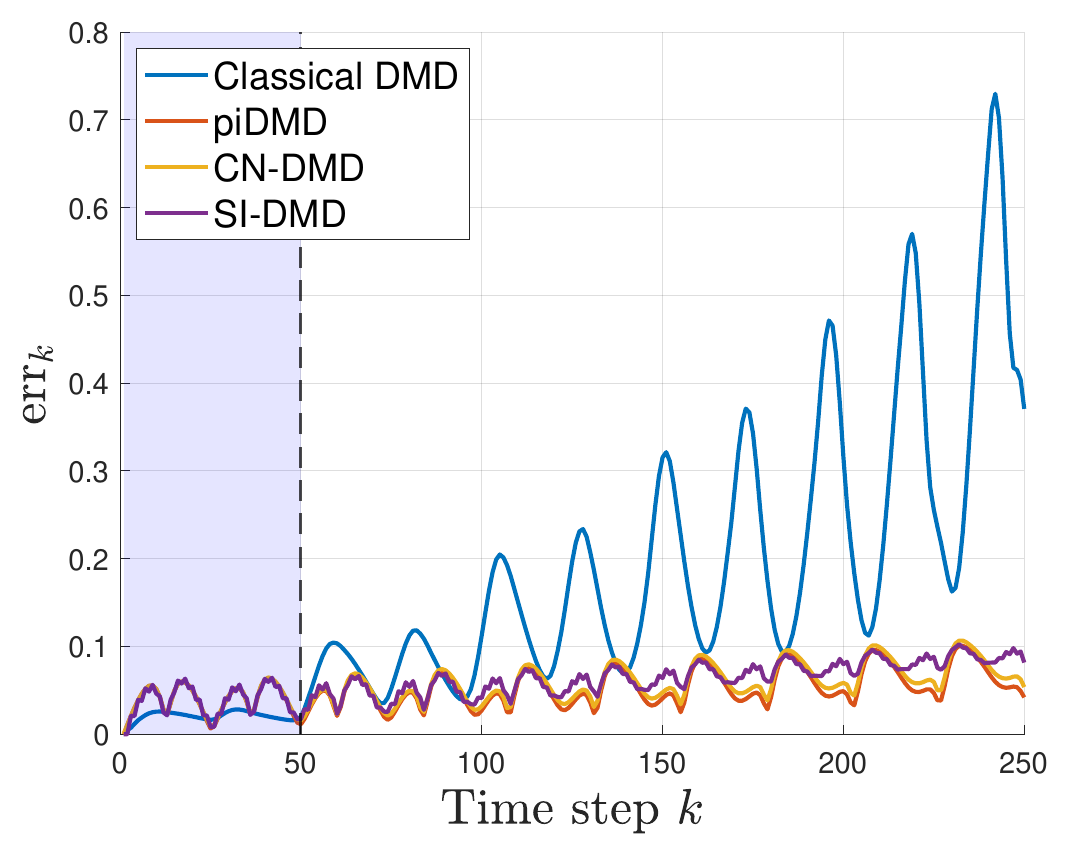}
\end{subfigure}
\begin{subfigure}[t]{0.48\textwidth}
    \includegraphics[width=\linewidth]{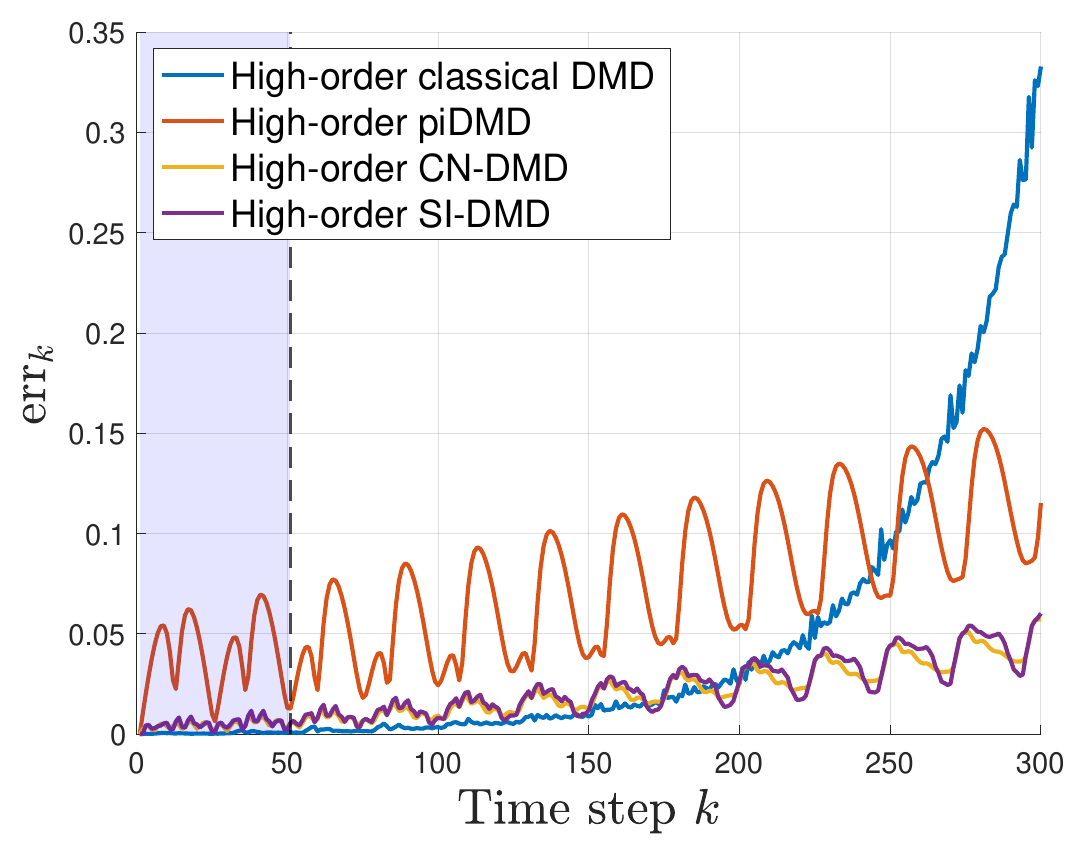}
\end{subfigure}
\caption{Comparison of DMD-type methods for the nonlinear \Sch\ equation~\eqref{eq:nonlinear_schrodinger_equation}.}
\label{fig:nonlinear_error}
\end{figure}

\section{Conclusions}
\label{sec:conclusion}
In this paper, we proposed two novel DMD-type methods, the {Crank--Nicolson} DMD (CN-DMD) and the semi-implicit DMD (SI-DMD), for learning and predicting the highly oscillatory dynamics of the semiclassical \Sch\ equation. Unlike classical DMD approaches, our methods focus on learning the underlying \Sch\ operator rather than the evolution operator. This innovation enables the systematic incorporation of both mass and energy conservation laws and further endows the resulting models with built-in model order reduction capabilities, without
the requirement for additional dimensionality-reduction preprocessing as in piDMD. 
Moreover, we provide preliminary theoretical estimates for both training and prediction errors, offering theoretical guarantees for the proposed methods. 
Extensive numerical experiments demonstrate the noise robustness, computational efficiency, and transferability to other equations of the proposed methods, indicating their potential as data-driven surrogates for traditional numerical PDE solvers in challenging tasks.


\bibliographystyle{siamplain}
\bibliography{references}
\end{document}